\documentclass[11pt,a4paper,reqno]{amsart}%

\usepackage{amsthm,amsmath,amsfonts,amssymb,amsxtra,appendix,bookmark,dsfont,latexsym}
\usepackage{amsfonts,mathrsfs}

\makeatletter
\usepackage{hyperref}
\usepackage{delarray,a4,color}
\usepackage{euscript}
\usepackage{bbm}
\usepackage[latin1]{inputenc}

\theoremstyle{plain}
\newtheorem{theorem}{Theorem}[section]
\newtheorem{lemma}[theorem]{Lemma}

\newtheorem{proposition}[theorem]{Proposition}

\theoremstyle{remark}

\numberwithin{equation}{section}

\DeclareMathOperator{\Tr}{Tr}
\DeclareMathOperator{\tr}{Tr}

\def\geqslant{\ge}
\def\leqslant{\le}

\def\nn{\nonumber}

\def\eps{\varepsilon}
\def\wto{\rightharpoonup}
\def\id{\mathbbm{1}}

\newcommand{\norm}[1]{\left\lVert #1 \right\rVert}

\newcommand\1{{\ensuremath {\mathds 1} }}

\newcommand{\brar}{\right|}
\newcommand{\bral}{\left\langle}
\newcommand{\ketr}{\right\rangle}
\newcommand{\ketl}{\left|}

\newcommand{\im}{\mathrm{i}}

\renewcommand{\epsilon}{\varepsilon}

\def\R {\mathbb{R}}

\def\cS {\mathcal{S}}

\def\cE {\mathcal{E}}

\def\R {\mathbb{R}}

\def\cS {\mathcal{S}}

\def\gH{\mathfrak{H}}
\def\b{|}

\renewcommand{\leq}{\leqslant}
\renewcommand{\geq}{\geqslant}

\newcommand{\bA}{\mathbf{A}}
\newcommand{\bB}{\mathbf{B}}
\newcommand{\be}{\mathbf{e}}
\newcommand{\bp}{\mathbf{p}}
\newcommand{\bJ}{\mathbf{J}}

\newcommand{\cR}{\mathcal{R}}
\newcommand{\bx}{\mathbf{x}}
 
\newcommand{\nablap}{\nabla^{\perp}}
\newcommand{\EAF}{E ^{\mathrm{af}}}
\newcommand{\cEAF}{\cE ^{\mathrm{af}}}
\newcommand{\uAF}{u ^{\mathrm{af}}}
\newcommand{\curl}{\mathrm{curl}}

\newcommand{\sym}{\mathrm{sym}}

\newcommand{\cDAF}{\mathscr{D}^{\mathrm{af}}}
\newcommand{\cMAF}{\mathcal{M}^{\mathrm{af}}}

\title[Average field approximation]{Average field approximation for almost bosonic anyons in a magnetic field}

\author[T.Girardot]{Th\'{e}otime Girardot}

\address{Universit\'e Grenoble Alpes \& CNRS, LPMMC (UMR 5493), B.P. 166, F-38042 Grenoble, France}
\email{theotime.girardot@lpmmc.cnrs.fr}

\date{April 27, 2020}

\begin{document}

\begin{abstract}
We study the ground state of a 
large number $N$ of 2D extended anyons in an external magnetic field. We consider a scaling limit where the statistics 
parameter $\alpha$ is proportional to $N ^{-1}$ when $N\to \infty$
which allows the statistics to be seen as a``perturbation around the bosonic end''.
Our model is that of bosons in a magnetic field and interacting 
through long-range magnetic potential generated by 
magnetic charges carried by each particle, smeared over discs of radius $R$. 
Our method allows to take $R\to 0$ not too fast at the same time as 
$N\to \infty$ : $R=N^{-1/4+\epsilon}$. We use the information theoretic version of the de Finetti theorem of 
Brand\~{a}o and Harrow to justify the so-called 
``average field approximation'': the particles behave like independent, identically distributed 
bosons interacting via a self-consistent magnetic field.  
\end{abstract}
\maketitle
\setcounter{tocdepth}{2}

 
\makeatletter
\def\@tocline#1#2#3#4#5#6#7{\relax
  \ifnum #1>\c@tocdepth 
  \else
    \par \addpenalty\@secpenalty\addvspace{#2}%
    \begingroup \hyphenpenalty\@M
    \@ifempty{#4}{%
      \@tempdima\csname r@tocindent\number#1\endcsname\relax
    }{%
      \@tempdima#4\relax
    }%
    \parindent\z@ \leftskip#3\relax \advance\leftskip\@tempdima\relax
    \rightskip\@pnumwidth plus4em \parfillskip-\@pnumwidth
    #5\leavevmode\hskip-\@tempdima
      \ifcase #1
       \or\or \hskip 1em \or \hskip 2em \else \hskip 3em \fi%
      #6\nobreak\relax
      \dotfill
      \hbox to\@pnumwidth{\@tocpagenum{#7}}
    \par
    \nobreak
    \endgroup
  \fi}
\makeatother
\tableofcontents

\section{Introduction}

This paper is a continuation of the quantum de Finetti-based approach to the mean-field
limit of anyonic grounds states of \cite{LunRou} with the motivation to extend its results 
in two different ways. The first extension concerns the physical situation.
Indeed we take into account an external magnetic field $\bB_{e}$ absent in \cite{LunRou}. The second concerns the improvement of the range of valididy
of the average field approximation thanks to an information-theoretic variant of the quantum de Finetti theorem due to Brand\~{a}o and Harrow \cite{BraHar-12,LiSmi-15}, and revisited in \cite{Rou}.\newline

In lower dimensions there are possibilities for quantum statistics other than the bosonic one and the fermionic one, so called intermediate or fractional statistics. 
Such particles, termed anyons (as in \emph{any}thing 
in between bosons and fermions), could arise as effective quasiparticles in 
many-body quantum systems confined to lower dimensions. For an example of these intermediate exchange statistic properties see \cite{LunSol-13b,LunSol-14} or \cite{LunSol-13a}.
These quasiparticles have been 
conjectured~\cite{AroSchWil-84} to be relevant 
for the fractional quantum Hall effect 
(see~\cite{Goerbig-09,Laughlin-99,LunRou-03} for review) and that is why we study the behavior of their ground state in a magnectic field.

\subsection{The model}
To describe anyons, the ``wave function'' has to (formally) behave as
$$\tilde{\Psi}(\mathbf{x}_{1},...,\mathbf{x}_{j},...,\mathbf{x}_{k},...,\mathbf{x}_{N})=e^{\im\alpha\pi}\tilde{\Psi}(\mathbf{x}_{1},...,\mathbf{x}_{k},...,\mathbf{x}_{j},..., \mathbf{x}_{N})$$
with $\alpha\in[0,1]$ and we can obtain this by writing:
$$\tilde{\Psi}(\mathbf{x}_{1},..., \mathbf{x}_{N})=\prod_{j<k}e^{\im\alpha\phi_{jk}}\Psi(\mathbf{x}_{1},...,\mathbf{x}_{N}) \;\:\text{where}\;\:
\phi_{jk}=\arg\frac{\mathbf{x}_{j}-\mathbf{x}_{k}}{|\mathbf{x}_{j}-\mathbf{x}_{k}|}$$
with $\Psi$ a bosonic wave function, symmetric under particle exchange. Note that cases $\alpha=1$ and $\alpha=0$ are respectively
the fermionic and the bosonic one. Let us now consider an usual Hamiltonian with a magnetic field $\bB_{e}=\curl\; \bA_{e}$ and a trapping potential $V$: 
$$\mathcal{H}_{N}=\sum_{j=1}^{N}(-\im\nabla_{j} + \bA_{e}(\mathbf{x}_{j}))^{2}+V(\mathbf{x}_{j})$$
with:
$$V(\bx) \ge c |\bx|^s - C, \quad s>0.$$
One can show that acting with it on $\tilde{\Psi}$ is equivalent to acting on $\Psi$ with the Hamiltonian:
\begin{equation}
 H_{N}=\sum_{j=1}^{N}\left((\mathbf{p}_{j}+\bA_{e}(\mathbf{x}_{j})+\alpha \bA(\mathbf{x}_{j}))^{2}+V(\mathbf{x}_{j})\right)
 \label{Ham}
\end{equation}
where, denoting $\bx^{\perp}=(x,y)^{\perp}=(-y,x)$,
\begin{equation}
 \mathbf{p}_{j}=-\im\nabla_{j}\:\:\text{and}\:\:\bA(\mathbf{x}_{j})=\sum_{j\neq k}\frac{(\mathbf{x}_{j}-\mathbf{x}_{k})^{\perp}}{|\mathbf{x}_{j}-\mathbf{x}_{k}|^{2}}.
\end{equation}
The operator $\bA(\mathbf{x}_{j})$ is the statistical gauge vector potential felt by the particle $j$ due to the 
influence of all the other particles. 
The statistics parameter is denoted by $\alpha$, 
corresponding to a statistical phase $e^{i\alpha\pi}$ under a continuous
simple interchange of two particles. In this so-called ``magnetic gauge picture'', 2D anyons are thus described 
as bosons, each of them carrying an Aharonov-Bohm magnetic flux of strength $\alpha$.
We henceforth work with the Hamiltonian $H_{N}$ acting on symmetric wave functions $$\Psi\in \gH = L^{2}_{\mathrm{sym}}(\mathbb{R}^{2N}).$$
In this work we will assume that $\alpha\to 0$ when $N\to \infty$, a limit that we call ``almost bosonic''.
More precisely we define a fixed constant $\beta$ such that
\begin{equation}
 \alpha=\frac{\beta}{N-1}\xrightarrow[N\to \infty]{}0.
 \label{alpha}
\end{equation}
In this way, the anyon statistics has a leading order effect manifest through \textit{the average field approximation}.

\subsection{Extended anyons}
The Hamiltonian \eqref{Ham} is too singular to be considered as acting on a pure tensor product $u^{\otimes N} \in \bigotimes_\sym^N \gH$, and, consequently, to consider a mean field limit. In order to circumvent this problem we will introduce a length $R$ over which the equivalent magnetic charge
is smeared. In our approach, $R\to0$ will make us recover the point-like anyons point of view.\newline

Let us consider the 2D Coulomb potential generated by a unit charge smeared over the disc of radius $R$:
\begin{equation}
 w_{R}(\mathbf{x})=\left(\log\b \;.\;\b *\chi_{R}\right)( \mathbf{x}),\:\:\text{with the convention}\:\:w_{0}=\log\b \;.\;\b
 \label{wr}
\end{equation}
and $\chi_{R}( \mathbf{x})$ a positive regularizing function with a unit mass and built as follows
$$\chi_{R}=\frac{1}{R^{2}}\chi\left (\frac{.}{R}\right)\; \text{and}\; \int_{\mathbb{R}^{2}}\chi=1$$
such that
\begin{equation}
  \chi( \mathbf{x})= 
   \begin{cases}
1/\pi^{2}\;\; \b  \mathbf{x}\b \le 1  \\
0\;\;\b  \mathbf{x}\b \ge 2
   \end{cases}
\end{equation}
being smooth, positive and decreasing between $1$ and $2$.
Since the magnetic field associated to $\bA(\mathbf{x}_{j})$ is:
$$\curl \bA(\mathbf{x}_{j})=2\pi\sum_{k\neq j}\delta(\mathbf{x}_{j}-\mathbf{x}_{k})$$
we will recover the same magnetic field in a distributional sense in the limit $R\to 0$ defining the magnetic potential vector:
\begin{equation}
 \bA^{R}(\mathbf{x}_{j})=\sum_{k\neq j}\nabla^{\perp}w_{R}(\mathbf{x}_{j}-\mathbf{x}_{k})
 \label{AJR}
\end{equation}
leading to the regularized Hamiltonian
\begin{equation}
 H_{N}^{R}=\sum_{j=1}^{N}\left(\left(\mathbf{p}_{j}+\bA_{e}+\alpha \bA^{R}(\mathbf{x}_{j})\right)^{2}+V(\mathbf{x}_{j})\right).
 \label{HRN}
\end{equation}
We denote
\begin{equation}
E^{R}(N)=\inf \sigma\left(H^{R}_{N} \right) 
\end{equation}
with $H^{R}_{N}$ acting on $L^{2}_{\mathrm{sym}}(\mathbb{R}^{2N})$.

For $R>0$ this operator is essentially self-adjoint on $L^{2}_{\mathrm{sym}}(\mathbb{R}^{2N})$ (see \cite[Theorem X.17]{ReeSim2} and \cite{AvrHerSim}). It does however not have a unique limit as $R\to 0$ and the Hamiltonian at $R=0$ (discussed e.g. in \cite{LunSol-14}) is \emph{not} essentially self-adjoint (see for instance \cite{CorOdd,DabSto,AdaTet,BouSor}). Nevertheless, in the joint limit $R\to 0$ and $N\to \infty $ we recover a unique well-defined (non-linear) model.
This regularization method is a generalization of the one used in \cite[equation 1.4]{LunRou} or in \cite{LarLun} where many-body magnetic Hardy inequalities and local exclusion principles are obtained. There, $$\chi=\frac{\1_{B(0,R)}}{\pi R^2} (x)$$ was used but we will need a smoother $\chi$, namely, $$\int_{\mathbb{R}^{2}}\b\widehat{\chi}(\mathbf{p})\b\mathrm{d}\bp<\infty $$
with $\widehat{\chi}$ the Fourier tranform of $\chi$.\newline
We expand $\eqref{HRN}$ in four terms that we will treat one by one:
\begin{align}
 H_{R}^{N}=&\sum_{j=1}^{N}\left((\mathbf{p}_{j}^{\bA})^{2}+V(\mathbf{x}_{j})\right) \;\text{``Kinetic and potential terms"}\nn\\
 &+\alpha\sum_{j\neq k}\left(\mathbf{p}_{j}^{\bA}.\nabla^{\perp}w_{R}(\mathbf{x}_{j}-\mathbf{x}_{k})+\nabla^{\perp}w_{R}(\mathbf{x}_{j}-\mathbf{x}_{k}).\mathbf{p}_{j}^{\bA}\right)\;\text{``Mixed two-body term"}\nonumber\\
 &+\alpha^{2}\sum_{j\neq k\neq l}\nabla^{\perp}w_{R}(\mathbf{x}_{j}-\mathbf{x}_{k}).\nabla^{\perp}w_{R}(\mathbf{x}_{j}-\mathbf{x}_{l})\;\text{``Three-body term"}\nonumber\\
 &+\alpha^{2}\sum_{j\neq k}\left|\nabla^{\perp}w_{R}(\mathbf{x}_{j}-\mathbf{x}_{k})\right|^{2}\;\text{``Singular two-body term"}.
 \label{expanded_H}
 \end{align}

\subsection{Average field approximation}
The model we have built until now is still hard to study. 
That is why we will set:
\begin{equation}
 \bA[\rho]=\nabla^{\perp}w_{0}*\rho\;\;\text{and}\;\;\bA^{R}[\rho]=\nabla^{\perp}w_{R}*\rho
\end{equation}
which makes $\bA$ independent of the precise positions $\mathbf{x}_{j}$. 
Here $\rho$ is a fixed one-body density normalized in $L^{1}(\mathbb{R}^{2})$.

We obtain from this the average field $N$-body Hamiltonian: 
\begin{equation}
 H^{\mathrm{af}}_{N}=\sum_{j=1}^{N}\left((\mathbf{p}_{j}^{\bA}+(N-1)\alpha \bA^{R}[\rho])^{2}+V(\mathbf{x}_{j})\right)
 \label{Haf}
\end{equation}
denoting: $$\mathbf{p}_{j}^{\bA}=\mathbf{p}_{j}+\bA_{e}(\mathbf{x}_{j}).$$
This average field approximation Hamiltonian is the usual way to describe almost-bosonic anyons. It has also been studied in \cite{CorLunRou-16,CorDubLunRou-19} with, in these cases, $\beta$ as a parameter.
In this way, the magnetic field becomes:
$$\bB(\bx)=\curl \beta \bA[\rho](\bx)=2\pi\beta \rho(\bx).$$
We can now consider the usual mean field ansatz taking $\Psi(\mathbf{x}_{1},..., \mathbf{x}_{N})=u^{\otimes N}$ with the consistency equation:
$$\rho(\mathbf{x}_{1})=\int_{\mathbb{R}^{2(N-1)}}\left|\Psi(\mathbf{x}_{1},\mathbf{x}_{2},..., \mathbf{x}_{N})\right|^{2}\mathrm{d}\mathbf{x}_{2}...\mathrm{d} \mathbf{x}_{N}$$
which means $|u|^{2}=\rho$ and
get the one particle average-field energy functional:
\begin{equation}
 \cEAF_{R}[u]=\int_{\mathbb{R}^{2}}\left(\left|(-\im\nabla+\bA_{e}+\beta \bA^{R}[|u|^{2}])u\right|^{2}+V|u|^{2}\right)
 \label{Eaf}
\end{equation}
with
\begin{equation}
\EAF_{R}=\inf \cEAF_{R}[u]
\end{equation}
under the unit mass constraint:
$$\int_{\mathbb{R}^{2}}|u|^{2}=1.$$
Alternatively, we can write
$$ \cEAF_{R}[u]=\int_{\mathbb{R}^{2}}\left(|\bp^{\bA}u|^{2}+V|u|^{2}+2\beta \bA^{R}[|u|^{2}]\cdot\left(\frac{\im}{2}(u\nabla\overline{u}-\overline{u}\nabla u)+\bA_{e}|u|^{2}+\frac{\beta}{2}\bA^{R}[|u|^{2}]|u|^{2}\right)\right)\mathrm{d}\bx$$
Our aim is now well defined and consists in proving:
\begin{equation}
 \frac{\inf \bral{\Psi_{N}},H_{N}^{R}{\Psi_{N}}\ketr}{N}\approx \inf\left\{ \cEAF_{R}[u],\int \b u \b^{2}=1\right\}
\end{equation}
 when $N\to \infty$.  
In other words, the average mean field energy is a good approximation for the ground state energy of $H_{N}^{R}$ (see Proposition $\eqref{prop:af_minimizer}$ in the appendix for a discussion about the existence of such minimizers).
\subsection{Main results}
We state our main theorem, justifying the average field approximation 
in the almost-bosonic limit at the level of the ground state. For technical reasons we assume that the one-body potential is confining 
\begin{equation}\label{eq:trap pot}
		V(\bx) \ge c |\bx|^s - C, \quad s>0,
\end{equation}
and
\begin{align}
\bA_{e}\in L^{2}_{loc}(\R^{2}),\;\;
\curl \bA_{e}=\bB_{e}=\bB_{0}+\tilde{\bB},\;\;
\bB_{0}=\mathrm{Cst},\;\;
\tilde{\bB}\in  W^{1,2+\epsilon}.
\end{align} 
We also need that the size $R$ of the extended anyons does not go to zero too fast in 
the limit $N\to \infty$ i.e $R= N^{-\eta}$ for $\eta$ having to be determined.

\begin{theorem}[\textbf{Validity of the average field approximation}]\label{thm:main ener}\mbox{}\\
Assume that we have $N$ extended anyons of radius
	$R = N ^{-\eta}$ for some 
	\begin{equation}\label{eq:eta restriction}
		0 < \eta <\frac{1}{4}
	\end{equation}
	and with the statistics parameter 
	$$\alpha = \frac{\beta}{(N-1)}$$ 
	for fixed $\beta \in \R$.
	Then, in the limit $N\to \infty$ we have 
	for the ground-state energy
	\begin{equation}\label{eq:main ener}
		\frac{E_R(N)}{N} \to \EAF_{0}= \EAF.
	\end{equation}
	Moreover, if $\Psi_N$ is a sequence of ground states for $H_N ^R$, with 
	associated reduced density matrices $$\gamma_N ^{(k)} := \tr_{k+1 \to N} \left[ |\Psi_N\rangle \langle \Psi_N|\right]$$ then 
	along a subsequence we have
	\begin{equation}\label{eq:main state}
		\gamma_N ^{(k)} \to \int_{\cMAF} |u ^{\otimes k} \rangle \langle u ^{\otimes k} | \,\mathrm{d}\mu(u) 
	\end{equation}
	strongly in trace class norm when $N\to \infty$, where $\mu$ is a Borel 
	probability measure supported on the set of minimizers of $\cEAF$,
	$$
		\cMAF := \lbrace u\in L ^2 (\R ^2) : \norm{u}_{L^2} = 1, \: \cEAF [u] = \EAF \rbrace.
	$$

\end{theorem}There are two main improvements in this theorem compared with \cite[Theorem 1.1]{LunRou}. \newline  
The first is the removal of the constraint 
$$\eta < \eta _{0}=\frac{1}{4}\left (1+\frac{1}{s}\right )^{-1}$$  providing us a less restritive result holding for every $s>0$. We achieve this using the method of \cite{Rou} which revisits the quantum de Finetti theorem due to Brand\~{a}o and Harrow  \cite{BraHar-12}. This implies to write the operators in 
$\eqref{HRN}$ in a tensorized form (see $\eqref{tensor_form}$). This forces us to use a smooth $\chi$ complicating the estimate of the $3$-body term (third line in $\eqref{expanded_H}$).

The second improvement is the addition of an external magnetic term $\bA_{e}$ in $H_{N}$. Transforming $\mathbf{p}_{j}^{0}$ in $\mathbf{p}_{j}^{\bA}$ the estimates of the cross product term becomes more difficult. The non commutativity $$\left [\left (p_{j}^{\bA}\right )_{1},\left (p_{j}^{\bA}\right )_{2}\right ]\neq 0$$ causes this difficulty.\newline

The needed new bounds will be the aim of the second section. Indeed, we will obtain a control of the Hamiltonian terms as a function of the kinetic energy in the limit $R\to\infty$.\newline In the third section we combine these results with an a priori bound on the kinetic energy also derived in Section 2 and the de Finetti theorem to get a lower bound on $E^{R}(N)/N$ as $\EAF$ plus an error depending on $R$, $N$ and a kinetic energy cut-off $\Lambda$ that we finally optimize in the limit $N\to\infty$. We also prove the convergence of states in this section.\newline We finally derive some auxiliary results such as the convergence $\EAF_{R}\to \EAF$ or the existence of minimizers for $\cEAF$ in the appendix.

\bigskip

\noindent\textbf{Acknowledgments.} 
I want to thank my thesis advisor Nicolas Rougerie for discussions, explanations and help that made this work possible.
I would also like to thank Alessandro Olgiati for inspiring discussions. Funding from the European Research Council (ERC) under the European Union's Horizon 2020 Research and Innovation Programme (Grant agreement CORFRONMAT No 758620) is gratefully acknowledged.

\section{The extended anyon Hamiltonian}
In this section we establish some bounds allowing to control the Hamiltonian $\eqref{HRN}$. 
We can expand it as in $\eqref{expanded_H}$ (and $\eqref{energy}$ for the energy expression) in order to estimate it term by term.

Because of the boundness of the interaction at fixed $R$,
$H_{N}^{R}$ is defined uniquely as a self-adjoint operator on $L^{2}_{\mathrm{sym}}(\R^{2N})$ with the same form domain as the non-interacting bosonic Hamiltonian, (see \cite[Proposition 7.20]{LieLos-01} for a definition of $H^{1}_{\bA_{e}}$):
$$\sum_{j=1}^{N}\left (\left (\mathbf{p}^{\bA}_{j}\right )^{2}+V(\mathbf{x}_{j})\right),\;\; \bA_{e}\in L^{2}_{\mathrm{loc}}.$$ 
As we will finally take the limit $R\to 0$ we need to deduce precise bounds depending on $R$ assuming
$R\ll R_{0}$, $R_{0}>0$ a fixed reference length scale. In the following, the generic constant $C$ may implicitly depends on $R_{0}$. Bounds are here similar to \cite{LunRou} but with 
$\mathbf{p}_{j}^{\bA}$ instead of $\mathbf{p}_{j}$ which complicates the proof of lemma \ref{lem:Mixed two-body term} concerning the mixed two-body term. As for the three-body term (lemma \ref{lem:three body}), the new regularizing function $\chi$ makes its proof based on the Hardy inequality harder.



\subsection{Operator bounds for the interaction terms}
Exploiting the regularizing effect of taking $R>0$ we will estimate the different terms in $\eqref{energy}$ .

\begin{lemma}[\textbf{The smeared Coulomb potential}]\label{lem:Coulomb_potential}\mbox{}\\
Let $w_{R}$ be defined as in $\eqref{wr}$. There is a constant $C>0$ depending only on $R_{0}$ such that
\begin{equation}
\sup_{\R^{2}}\left|\nabla w_{R}\right \b\le \frac{C}{R}\;\;\text{and}\;\;\b\b\nabla w_{R}\b\b_{L^{p}(\R^{2})}\le C_{p}R^{2/p-1}
\end{equation}
for any $2<p<\infty$.
\label{norme_wr}
\end{lemma}
\begin{proof}
We apply Newton's theorem \cite[Theorem 9.7]{LieLos-01} to calculate $w_{R}$ in each domain:
\begin{equation}
 \left|\nabla w_{R}(\bx)\right \b=
 \begin{cases}
1/\b \bx\b, \; \; \b \bx\b\in [2R,+\infty[\\
1/\b \bx\b\int_{B(0,\b \bx\b)}\chi_{R}(u)\mathrm{d}u,\;\;  \b \bx\b\in ]R,2R[\\
\b \bx\b /\pi R^{2}, \;\;  \b \bx\b\in [0,R]
  \end{cases}
   \label{nabla_wr}
\end{equation}For $\b \bx\b$ between $0$ and $2R$ the given formula comes from:
$$w_{R}=\log\b \bx\b \int_{0}^{2\pi}\int_{0}^{\b \bx\b}\chi_{R}(y)y \mathrm{d}y \mathrm{d}\theta + \int_{0}^{2\pi}\int_{\b \bx\b}^{2R}\log(y)\chi_{R}(y) y \mathrm{d}y \mathrm{d}\theta $$
which gives the first conclusion of the lemma.

To obtain the second we need to compute
\begin{equation}
\b\b\nabla w_{R}\b\b^{p}_{L^{p}(\R^{2})}\le 2\pi\int_{0}^{R}\frac{r^{p}}{\pi^{p}R^{2p}}r\mathrm{d}r+2\pi\int_{R}^{+\infty}r^{-p}r\mathrm{d}r\le C_{p}^{p}R^{2-p}
\end{equation}
with $C_{p}>0$ only depending on $p>2$.
\end{proof}



\begin{lemma}[\textbf{Singular two-body term}]\label{lem:Singular two-body term}\mbox{}\\
 For any $\epsilon >0$ we have as multiplication operators on $L^{2}(\mathbb{R}^{4})$:
 \begin{equation}
  \left|\nabla w_{R}(\mathbf{x}_{1}-\mathbf{x}_{2})\right|^{2}\leqslant C_{\epsilon}R^{-\epsilon}\left((\mathbf{p}_{1}^{\bA})^{2}+1\right)
  \label{WRcarre}
 \end{equation}
  \begin{equation}
  \left|\nabla w_{R}(\mathbf{x}_{1}-\mathbf{x}_{2})\right|^{2}\leqslant C_{\epsilon}R^{-\epsilon}\left(\mathbf{p}_{1}^{2}+1\right)
  \label{WRcarre2}
 \end{equation}
\end{lemma}

\begin{proof}
 Taking any $w$: $\mathbb{R}^{2}\mapsto\mathbb{R}$ and $f\in C^{\infty}_{c}(\R^{4})$ with $u=\b f \b$, we compute
 \begin{align*}
  \bral{f},w(\bx_{1}-\bx_{2}){f}\ketr &=\bral{u},w(\bx_{1}-\bx_{2}){u}\ketr\\
  &=\iint_{\mathbb{R}^{2}\times\mathbb{R}^{2}}w(\bx_{1}-\bx_{2})\b u(\bx_{1},\bx_{2})\b^{2}\mathrm{d}\bx_{1}\mathrm{d}\bx_{2}\;\;\text{because $w$ does not see the phase}\\
   &\leqslant \int_{\mathbb{R}^{2}}\left(\int_{\mathbb{R}^{2}}|w(\bx_{1}-\bx_{2})|^{p}\mathrm{d}\bx_{1}\right)^{\frac{1}{p}}\left(\int_{\mathbb{R}^{2}}|u(\bx_{1},\bx_{2})|^{2q}\mathrm{d}\bx_{1}\right)^{\frac{1}{q}}\mathrm{d}\bx_{2}\;\;\text{by H\"{o}lder's inequality.}
   \end{align*}
Then, using  Sobolev's inequality \cite[Theorem 8.5 ii]{LieLos-01}
\begin{align*}    
 \bral{f},w(\bx_{1}-\bx_{2}){f}\ketr & \leqslant||w||_{L^{p}}\int_{\mathbb{R}^{2}}||u(.,\bx_{2})||^{2}_{L^{2q}}\mathrm{d}\bx_{2}\\
  &\leqslant C||w||_{L^{p}}\int_{\mathbb{R}^{2}}\left(||\nabla u(.,\bx_{2})||^{2}_{L^{2}}+||u(.,\bx_{2})||^{2}_{L^{2}}\right)\mathrm{d}\bx_{2}\\
  &= C||w||_{L^{p}}\iint_{\mathbb{R}^{2}\times\mathbb{R}^{2}}\left(|\nabla u(\bx_{1},\bx_{2})|^{2}+|u(\bx_{1},\bx_{2})|^{2}\right)\mathrm{d}\bx_{1}\mathrm{d}\bx_{2}\\ &\leqslant C||w||_{L^{p}}\bral{f},\left((\mathbf{p}_{x}^{\bA})^{2}+1\right){f}\ketr\;\;\text{using the diamagnetic inequality $\eqref{eq:af_diamagnetic_simple}$}
 \end{align*}
 with $p>1$ and $q=\frac{p}{p-1}$. We now use $\eqref{norme_wr}$ with $w=\left|\nabla w_{R}\right|^{2}$ and $p=1+\epsilon'$ to get
 $$||w||_{L^{p}}=||\nabla w_{R}||^{2}_{L^{2p}}\leqslant C^{2}_{2p}R^{2/p-2}\leqslant C_{\epsilon}R^{-\epsilon}$$
 with a constant $C_{\epsilon}>0$ for given $\epsilon >0$. The proof of the second inequality is the same without the diamagnetic inequality step.
\end{proof}

We next deal with the term mixing the position and the momentum. The first inequality we will get will have a bad $R$-dependence but a better behavior for large momenta than the other.



\begin{lemma}[\textbf{Mixed two-body term}]\label{lem:Mixed two-body term}\mbox{}\\
 Consider a magnetic field $$ \bB_{e} =\curl \bA_{e}= \bB _{0}+\tilde{ \bB }$$ with $ \bB _{0}=\mathrm{Cst}$ and $\tilde{\bB}\in W^{1,2+\epsilon}$.
 
For $R<R_{0}$ and $\epsilon > 0$ we have, as operators on $L^{2}_{sym}(\mathbb{R}^{4})$:
 \begin{equation}
  \left|\bp_{1}^{\bA}.\nabla^{\perp}w_{R}(\mathbf{x}_{1}-\mathbf{x}_{2})+\nabla^{\perp}w_{R}(\mathbf{x}_{1}-\mathbf{x}_{2}).\bp_{1}^{\bA}\right|\leqslant CR^{-1}|\bp_{1}^{\bA}|
  \label{wrp1}
  \end{equation}
  \begin{equation}
 \left|\bp^{\bA}_{1}.\nabla^{\perp}w_{R}(\mathbf{x}_{1}-\mathbf{x}_{2})+\nabla^{\perp}w_{R}(\mathbf{x}_{1}-\mathbf{x}_{2}).\bp^{\bA}_{1}\right|\leqslant C_{\epsilon}R^{-\epsilon}\left (\left(\mathbf{p}^{\bA}_{1}\right)^{2}+1\right )
 \label{TBT2}
\end{equation}
\begin{equation}
 \left|\bp_{1}.\nabla^{\perp}w_{R}(\mathbf{x}_{1}-\mathbf{x}_{2})+\nabla^{\perp}w_{R}(\mathbf{x}_{1}-\mathbf{x}_{2}).\bp_{1}\right|\leqslant C_{\epsilon}R^{-\epsilon}(\mathbf{p}_{1}^{2}+1)
\label{TBT3}
\end{equation}

  \end{lemma}

 \begin{proof}
 Proof of $\eqref{wrp1}$:\newline
  We first note that $\nabla_{\bx_{1}}.\nabla^{\perp}w_{R}(\mathbf{x}_{1}-\mathbf{x}_{2})=0$, consequently
  $$\bp_{1}^{\bA}.\nabla^{\perp}w_{R}(\mathbf{x}_{1}-\mathbf{x}_{2})=\nabla^{\perp}w_{R}(\mathbf{x}_{1}-\mathbf{x}_{2}).\bp_{1}^{\bA}$$
  and we square the expression we want to estimate
  $$\left(\mathbf{p}_{1}^{\bA}.\nabla^{\perp}w_{R}(\mathbf{x}_{1}-\mathbf{x}_{2})+\nabla^{\perp}w_{R}(\mathbf{x}_{1}-\mathbf{x}_{2}).\bp_{1}^{\bA}\right)^{2}=4\bp_{1}^{\bA}.\nabla^{\perp}w_{R}(\mathbf{x}_{1}-\mathbf{x}_{2})\nabla^{\perp}w_{R}(\mathbf{x}_{1}-\mathbf{x}_{2}).\bp_{1}^{\bA}.$$
  Then for any $f(\bx_{1},\bx_{2})\in C_{c}^{\infty}(\mathbb{R}^{4})$
  $$\left|\bral{f},\left(\mathbf{p}_{1}^{\bA}.\nabla^{\perp}w_{R}(\mathbf{x}_{1}-\mathbf{x}_{2})+\nabla^{\perp}w_{R}(\mathbf{x}_{1}-\mathbf{x}_{2}).\bp_{1}^{\bA}\right)^{2}{f}\ketr\right|$$
  $$\leqslant 4\iint_{\mathbb{R}^{2}\times\mathbb{R}^{2}}\left|\bp_{1}^{\bA}f(\bx_{1},\bx_{2})\right|^{2}\left|\nabla^{\perp}w_{R}(\mathbf{x}_{1}-\mathbf{x}_{2})\right|^{2}\mathrm{d}\bx_{1}\mathrm{d}\bx_{2}$$
  using $\eqref{WRcarre}$ we get
  $$\left|\bral{f},\left(\mathbf{p}_{1}^{\bA}.\nabla^{\perp}w_{R}(\mathbf{x}_{1}-\mathbf{x}_{2})+\nabla^{\perp}w_{R}(\mathbf{x}_{1}-\mathbf{x}_{2}).\bp_{1}^{\bA}\right)^{2}{f}\ketr\right|\leqslant \frac{C}{R^{2}}\iint_{\mathbb{R}^{2}\times\mathbb{R}^{2}}\left|\bp_{1}^{\bA}f(\bx_{1},\bx_{2})\right|^{2}\mathrm{d}\bx_{1}\mathrm{d}\bx_{2}$$
    $$\left(\mathbf{p}_{1}^{\bA}.\nabla^{\perp}w_{R}(\mathbf{x}_{1}-\mathbf{x}_{2})+\nabla^{\perp}w_{R}(\mathbf{x}_{1}-\mathbf{x}_{2}).\bp_{1}^{\bA}\right)^{2}\leqslant \frac{C}{R^{2}}\left(\mathbf{p}_{1}^{\bA}\right)^{2}$$
  and we deduce $\eqref{wrp1}$ because the square root is operator monotone, 
  see \cite[Chapter 5]{Bhatia}.\newline

For the proof of $\eqref{TBT2}$, writing $\bx=(x,y)$ the coordinates of the particle located in $\bx$, we start from
 \begin{align*}
T=&\left|\bral{f},\left(\mathbf{p}^{\bA}_{1}.\nabla^{\perp}w_{R}(\mathbf{x}_{1}-\mathbf{x}_{2})+\nabla^{\perp}w_{R}(\mathbf{x}_{1}-\mathbf{x}_{2}).\bp^{\bA}_{1}\right)^{2}{f}\ketr\right|\\
&\leqslant 4\iint_{\mathbb{R}^{2}\times\mathbb{R}^{2}}\left|\bp^{\bA}_{1}f(\bx_{1},\bx_{2})\right|^{2}\left|\nabla^{\perp}w_{R}(\mathbf{x}_{1}-\mathbf{x}_{2})\right|^{2}\mathrm{d}\bx_{1}\mathrm{d}\bx_{2}\\
&= 4\iint_{\mathbb{R}^{2}\times\mathbb{R}^{2}}(\left|p^{\bA}_{x}f(\bx_{1},\bx_{2})\right|^{2}+\left|p^{\bA}_{y}f(\bx_{1},\bx_{2})\right|^{2})\left|\nabla^{\perp}w_{R}(\mathbf{x}_{1}-\mathbf{x}_{2})\right|^{2}\mathrm{d}\bx_{1}\mathrm{d}\bx_{2}\\
&=4\left(\bral{p^{\bA}_{x}f},{\left|\nabla^{\perp}w_{R}(\mathbf{x}_{1}-\mathbf{x}_{2})\right|^{2}p^{\bA}_{x}f}\ketr+\bral{p^{\bA}_{y}f},{\left|\nabla^{\perp}w_{R}(\mathbf{x}_{1}-\mathbf{x}_{2})\right|^{2}p^{\bA}_{y}f}\ketr\right)\\
&\leqslant \frac{C_{\epsilon}}{R^{\epsilon}}\left(\bral{p^{\bA}_{x}f},{((\mathbf{p}^{\bA}_{1})^{2}+1)p^{\bA}_{x}f}\ketr+\bral{p^{\bA}_{y}f},{((\mathbf{p}^{\bA}_{1})^{2}+1)p^{\bA}_{y}f}\ketr\right)
\end{align*}
Using the bound $\eqref{WRcarre}$.\newline 
In order to recover $(\mathbf{p}^{\bA}_{x,y})^{2}$ we have to calculate commutators involving $p^{\bA}_{x}$ and $\left (\mathbf{p}^{\bA}_{1}\right )^{2}$.
\begin{align*}
T&\leqslant \frac{C_{\epsilon}}{R^{\epsilon}}\bral{f},{((\mathbf{p}^{\bA}_{1})^{2}+1)\left(\mathbf{p}^{\bA}_{1}\right)^{2}f}\ketr+ \frac{C_{\epsilon}}{R^{\epsilon}}\bral{f},{\left([p^{\bA}_{x},\left(\mathbf{p}^{\bA}_{y}\right)^{2}]p^{\bA}_{x}+[p^{\bA}_{y},\left(\mathbf{p}^{\bA}_{x}\right)^{2}]p^{\bA}_{y}\right)f}\ketr\\
&\leqslant\frac{C_{\epsilon}}{R^{\epsilon}}\bral{f},{\left((\mathbf{p}^{\bA}_{1})^{2}+1\right)^{2}f}\ketr+\frac{C_{\epsilon}}{R^{\epsilon}}\bral{f},F{f}\ketr
\end{align*}
To control $F$ we calculate:
\begin{align*}
[p_{x}^{\bA},p_{y}^{\bA}]=-\im\left(\nabla\land \bA_{e}\right)=-\im \bB_{e}  
\end{align*}
Note that we are in two dimensions and that is why $\bB_{e}$ is a scalar (physically carried by the axis perpendicular to the plane). We compute:
$$[p_{x}^{\bA},-\im \bB_{e} ]=-\partial_{x} \bB_{e} \;\;\text{and}\;\;[p_{y}^{\bA},-\im \bB_{e} ]=\partial_{y} \bB_{e}. $$
Thus we obtain:
\begin{align*}
 F&=p^{\bA}_{y}(-\im \bB_{e} )p^{\bA}_{x}-\im \bB_{e}  p^{\bA}_{y}p^{\bA}_{x}-p^{\bA}_{x}(-\im \bB_{e} )p^{\bA}_{y}+\im \bB_{e} p^{\bA}_{x}p^{\bA}_{y}\\
 &=(\partial_{y} \bB_{e} -\im \bB_{e} p^{\bA}_{y})p^{\bA}_{x}-\im \bB_{e} p^{\bA}_{y}p^{\bA}_{x}-(\partial_{x} \bB_{e} -\im \bB_{e} p^{\bA}_{x})p^{\bA}_{y}+\im \bB_{e} p^{\bA}_{x}p^{\bA}_{y}\\
 &=\partial_{y} \bB_{e} p^{\bA}_{x}+\partial_{x} \bB_{e} p^{\bA}_{y}-2\im \bB_{e} (p^{\bA}_{y}p^{\bA}_{x}-p^{\bA}_{x}p^{\bA}_{y})\\
 &=2  \bB_{e} ^{2}-(\partial_{y}  \bB_{e} )p_{x}^{\bA}-(\partial_{x}  \bB_{e} )p_{y}^{\bA}\\ 
 &\leqslant 2   \bB_{e} ^{2}+\frac{1}{2}\left((\partial_{x}   \bB_{e} )^{2}+(\partial_{y}   \bB_{e} )^{2}+(p_{x}^{\bA})^{2}+(p_{y}^{\bA})^{2}\right)\\
 &=2   \bB_{e} ^{2}+\frac{1}{2}\b\nabla   \bB_{e} \b^{2}+\frac{1}{2}\b\mathbf{p}^{\bA}\b^{2}
\end{align*}
Here $\nabla   \bB_{e}$ is the vector $(\partial_{x}\bB_{e},\partial_{y}\bB_{e})$.
Splitting the magnetic field as $\bB_{e}=\bB_{0}+\tilde{\bB}$ with $\bB_{0}=\mathrm{Cst}$ we get for the $\bB_{e}^{2}=\bB_{0}^{2}+2\bB_{0}.\tilde{\bB}+\tilde{\bB}^{2}$ and the $(\nabla \bB_{e})^{2}=(\nabla \tilde{\bB})^{2}$ terms:
\begin{align*}
 &\bral{f},\bB_{e}^{2}f\ketr=\bB_{0}^{2}+2\bB_{0}\bral{f},\tilde{\bB}f\ketr+\bral{f},\tilde{\bB}^{2}f\ketr\\
 &\leqslant \bB_{0}^{2}+\left(2\bB_{0}||\tilde{\bB}||_{L^{p}}+||\tilde{\bB}||^{2}_{L^{2p}}\right)||f||^{2}_{L^{2q}}\leqslant C \bral{f},\left(\left(\mathbf{p}_{1}^{\bA}\right)^{2}+1\right)f\ketr\\
 &\bral{f},(\nabla \tilde{\bB})^{2}f\ketr\leqslant||(\nabla \tilde{\bB})^{2}||_{L^{p}}||f||^{2}_{L^{2q}}\leqslant C||(\nabla \tilde{\bB})^{2}||_{L^{p}} \bral{f},\left(\left(\mathbf{p}_{1}^{\bA}\right)^{2}+1\right)f\ketr
\end{align*}
using H\"{o}lder's inequality with $1=1/p+1/q$, that $\bB$ and $(\nabla \bB)^{2}$ are multiplicator operators and the Sobolev' inequality \cite[Theorem 8.5 ii]{LieLos-01}:
$$ ||g||^{2}_{L^{2q}}\leqslant C_{q}\left(||\nabla g||^{2}_{L^{2}}+||g||^{2}_{L^{2}}\right)$$
for $1\leqslant q \leqslant \infty$. \newline
So because we have $\tilde{\bB}\in W^{1,2+\epsilon}$:
\begin{align*}
 F&\leqslant C\left((\mathbf{p}_{1}^{\bA})^{2}+1\right)
\end{align*}
and we get
$$T\leqslant \frac{C_{\epsilon}}{R^{\epsilon}}\bral{f},{\left((\mathbf{p}^{\bA}_{1})^{2}+1\right)^{2}f}\ketr$$
$$\left(\mathbf{p}^{\bA}_{1}.\nabla^{\perp}w_{R}(\mathbf{x}_{1}-\mathbf{x}_{2})+\nabla^{\perp}w_{R}(\mathbf{x}_{1}-\mathbf{x}_{2}).\bp^{\bA}_{1}\right)^{2}\leqslant\frac{C_{\epsilon}}{R^{\epsilon}}\left((\mathbf{p}^{\bA}_{1})^{2}+1\right)^{2}$$
and the result follows  by taking the square root. The proof of the third inequality \eqref{TBT3} is a simpler version of what we just did.
\end{proof}


\begin{lemma}[\textbf{Three-body term}]\label{lem:three body}\mbox{}\\
 We have that, as multiplication operators on $L^{2}_{sym}(\mathbb{R}^{6})$:
 \begin{equation}
  \left|\nabla^{\perp}w_{R}(\mathbf{x}_{1}-\mathbf{x}_{2}).\nabla^{\perp}w_{R}(\mathbf{x}_{1}-\mathbf{x}_{3})\right|\leqslant C\left(\left(\mathbf{p}_{1}^{\bA}\right)^{2}+1\right)
  \label{W3C}
  \end{equation}
   \begin{equation}
  \left|\nabla^{\perp}w_{R}(\mathbf{x}_{1}-\mathbf{x}_{2}).\nabla^{\perp}w_{R}(\mathbf{x}_{1}-\mathbf{x}_{3})\right|\leqslant C\left(\mathbf{p}_{1}^{2}+1\right)
  \label{W3C2}
 \end{equation}
 \end{lemma}
To prove this lemma we will use the three-particle Hardy inequality \cite[Lemma 3.6]{HofLapTid-08} as done in \cite[Lemma 2.4]{LunRou} but with our new $\chi \neq 1\!\!1_{B(0,R)}$ which implies to take into account the sign of each considered term. The second difference is that we do not use Clifford algebra.

\begin{lemma}[\textbf{Three-body Hardy inequality},\cite{HofLapTid-08}.]\label{lem:Hardy}\mbox{}\\
Let $d=2$ and $u:\R ^{6} \to \mathbb{C}$. 
Let $\cR(\bx,\mathbf{y},\mathbf{z})$ be the circumradius of 
the triangle with vertices $\bx,\mathbf{y},\mathbf{z} \in \R^2$,
and $\rho(\bx,\mathbf{y},\mathbf{z}) := \sqrt{|\bx-\mathbf{y}|^2 + |\mathbf{y}-\mathbf{z}|^2 + |\mathbf{z}-\bx|^2}$. Then 
$\cR^{-2} \leqslant 9\rho^{-2}$ pointwise, and, for all $u\in H^{1}$
\begin{equation}\label{eq:Hardy}
3\int_{\R ^{6}} \frac{|u(\bx,\mathbf{y},\mathbf{z})| ^2}{\rho(\bx,\mathbf{y},\mathbf{z})^2} \mathrm{d}\bx\mathrm{d}\mathbf{y}\mathrm{d}\mathbf{z} \leqslant 
\int_{\R ^{6}} \left( \left| \nabla_\bx u \right| ^2 
+ \left| \nabla_\mathbf{y} u \right| ^2 
+ \left| \nabla_\mathbf{z} u \right| ^2\right) \mathrm{d}\bx \mathrm{d}\mathbf{y} \mathrm{d}\mathbf{z} 
	\end{equation}
\end{lemma}

\begin{proof}[Proof of Lemma~\ref{lem:three body}]
	Since we consider the multiplication operator as acting on symmetric wave functions it is 
	equivalent to estimate the sum:
	\begin{equation}\label{eq:cyclic}
		S=\sum_{\bx,\mathbf{y},\mathbf{z}} \nabla w_R (\bx-\mathbf{y}) . \nabla w_R (\bx-\mathbf{z}) =
		 \sum_{\bx,\mathbf{y},\mathbf{z}} \frac{\bx-\mathbf{y}}{|\bx-\mathbf{y}|} . \frac{\bx-\mathbf{z}}{|\bx-\mathbf{z}|}v(\bx-\mathbf{y})v(\bx-\mathbf{z}) 
		\end{equation}
	where $\bx,\mathbf{y},\mathbf{z}$ under the sum means cyclic in $\bx,\mathbf{y},\mathbf{z}$ and $v(u)=\partial_{u}w_{R}(u)$ is a radial function whose expression has been calculated in $\eqref{nabla_wr}$ such that
	\begin{align}
	 v\left(\b \bx\b \right )= \left \b\nabla w_{R}(\bx)\right \b=
 \begin{cases}
1/\b \bx\b, \; \; \b \bx\b\in [2R,+\infty[\\
1/\b \bx\b\int_{B(0,\b \bx\b)}\chi_{R}(u)\mathrm{d}u,\;\;  \b \bx\b\in ]R,2R[\\
\b \bx\b /\pi R^{2}, \;\;  \b \bx\b\in [0,R]
  \end{cases}
   \end{align}
In general, let each $\bx,\mathbf{y},\mathbf{z} \in \R^2$ denote the vertices of a triangle.
Then we claim the following geometric fact:
	\begin{equation}
	\label{eq:geom three body}
		\b S\b =\left| \sum_{\bx,\mathbf{y},\mathbf{z}} \frac{\bx-\mathbf{y}}{|\bx-\mathbf{y}|} . \frac{\bx-\mathbf{z}}{|\bx-\mathbf{z}|}v(\bx-\mathbf{y})v(\bx-\mathbf{z}) \right|
		\leqslant \frac{C}{\rho(\bx,\mathbf{y},\mathbf{z})^2}.
	\end{equation}
	where $C$ is a positive contant independent of $R$. Observe that for $R\leqslant |\bx|\leqslant 2R$, $v(\bx)\ge 0$ and that $$\frac{C_{1}}{R}\le v(\bx)\le\frac{C_{2}}{R}$$ 
where $C_{1}\le C_{2}$ are two non-negative constants independent of $R$.
We will proceed in several steps respectively considering the cases:\newline
\begin{enumerate}
\item All lengths of the triangle are larger than $2R$\newline
\item All lengths are smaller than $2R$\newline
\item Two of the lengths are smaller than $2R$ and one greater.\newline
\item Two of the lengths are larger than $2R$ and one smaller.
\end{enumerate}
and establish for each the inequality $S\leqslant C\rho^{-2}$. The second, $S\geqslant -C\rho^{-2}$ being easy or similar, depending on the case, will not be treated.
\medskip

Step 1.
All lengths of the triangle are greater than $2R$ i.e $|\bx|\ge 2R$. Then
\begin{equation}
 S=\sum_{\bx,\mathbf{y},\mathbf{z}} \frac{\bx-\mathbf{y}}{|\bx-\mathbf{y}|^{2}} . \frac{\bx-\mathbf{z}}{|\bx-\mathbf{z}|^{2}}
\end{equation}
We proceed as in \cite[Lemma 3.2]{HofLapTid-08}: $a=\bx-\mathbf{y}$, $b=\bx-\mathbf{z}$ and $\phi$ the angle between $a$ and $b$. A simple computation of $S$ leads to:
\begin{align*}
 S=\frac{a.b}{|a|^{2}|b|^{2}}-\frac{a}{|a|^{2}}.\frac{(b-a)}{|b-a|^{2}}-\frac{b}{|b|^{2}}.\frac{(a-b)}{|b-a|^{2}}\\
 \frac{2(|a|^{2}|b|^{2}-(a.b)^{2})}{|a|^{2}|b|^{2}|b-a|^{2}}=\frac{2\sin(\phi)^{2}}{|\mathbf{y}-\mathbf{z}|^{2}}=\frac{1}{2}R^{-2}(\bx,\mathbf{y},\mathbf{z})\leqslant \frac{9}{2}\rho^{-2}
 \end{align*}
by Lemma \ref{lem:Hardy}.\newline

Step 2.
Before initiating the proof we need to enlight some general facts about this geometric configuration.
Indeed we will need to know which term of the sum is positive or negative in each case we will study.
First, as the sum of the triangle's angles is $\pi$ we know that if one of them is larger than $\pi/2$ the others have to be smaller.
That implies that only one of the scalar products of the sum can be negative in each configuration and we are able to identify it.
For example if $\widehat{\bx\mathbf{y}\mathbf{z}}\ge \pi/2$ then the first term of the sum has to be negative and the two others positive.
If all of them are under $\pi/2$, then the sum only consists of positive terms. This case being easier and involving the same reasonning than the others we will systematically ignore it.\newline

Here all lengths are smaller than $2R$ and the negative term is the second.
This term can be bounded in three different ways. First:
$$\frac{(\mathbf{y}-\mathbf{z}).(\mathbf{y}-\bx)}{\b \mathbf{y}-\mathbf{z}\b \b \mathbf{y}-\bx\b}v(\mathbf{y}-\mathbf{z})v(\mathbf{y}-\bx)\le\frac{(\mathbf{y}-\mathbf{z}).(\mathbf{y}-\bx)}{\pi^{2}R^{4}}$$
 if $\b \mathbf{y}-\mathbf{z}\b$ and $\b \mathbf{y}-\bx\b$ are smaller than $R$ (here the inequality is actually an equality). Or it can be bounded as 
 $$\frac{(\mathbf{y}-\mathbf{z}).(\mathbf{y}-\bx)}{\b \mathbf{y}-\mathbf{z}\b \b \mathbf{y}-\bx\b}v(\mathbf{y}-\mathbf{z})v(\mathbf{y}-\bx)\le C_{1}\frac{(\mathbf{y}-\mathbf{z}).(\mathbf{y}-\bx)}{2R^{4}}\;\;$$
 if only $\b \mathbf{y}-\mathbf{z}\b$ or $\b \mathbf{y}-\bx\b$ is smaller than $R$ and by
 $$\frac{(\mathbf{y}-\mathbf{z}).(\mathbf{y}-\bx)}{\b \mathbf{y}-\mathbf{z}\b \b \mathbf{y}-\bx\b}v(\mathbf{y}-\mathbf{z})v(\mathbf{y}-\bx)\le\frac{C_{1}^{2}}{4}\frac{(\mathbf{y}-\mathbf{z}).(\mathbf{y}-\bx)}{R^{4}}$$
 if $\b \mathbf{y}-\mathbf{z}\b$ and $\b \mathbf{y}-\bx\b$ are between $R$ and $2R$.
Arguing similarly for the positive terms we get
\begin{align*}
 R^{4}S\le &\max\left(1,C_{2},C_{2}^{2}\right)(\bx-\mathbf{y}).(\bx-\mathbf{z})+\min\left(1,C_{1}/2,C_{1}^{2}/4\right)(\mathbf{y}-\mathbf{z}).(\mathbf{y}-\bx)\\
 &+\max\left(1,C_{2},C_{2}^{2}\right)(\mathbf{z}-\bx).(\mathbf{z}-\mathbf{y})
\end{align*}
\begin{align*}
 = &\max\left(1,C_{2},C_{2}^{2}\right)\left((\bx-\mathbf{y}).(\bx-\mathbf{z})+(\mathbf{y}-\mathbf{z}).(\mathbf{y}-\bx)+(\mathbf{z}-\bx).(\mathbf{z}-\mathbf{y})\right)\\
 &+ \left(\min\left(1,C_{1}/2,C_{1}^{2}/4\right)
 -\max\left(1,C_{2},C_{2}^{2}\right)\right)(\mathbf{y}-\mathbf{z}).(\mathbf{y}-\bx)
\end{align*}
and we obtain two positive constants $C_{3}$ and $C_{4}$ such that
\begin{equation}
 \left|S\right|\leqslant \frac{C_{3}}{R^{4}}\rho^{2}\leqslant C_{4}\rho^{-2}.
\end{equation}

Step 3.
We consider the case where two of the edges are short and one long. Let's take
$|\bx-\mathbf{y}|,|\mathbf{y}-\mathbf{z}|\le 2R$ and $|\bx-\mathbf{z}|\ge 2R$. In this configuration the negative term is the second. As we did in Step 2 we start by writing
\begin{align*}
 S_{3}\leqslant &\max(1,C_{2})\frac{(\bx-\mathbf{y}).(\bx-\mathbf{z})}{R^{2}|\bx-\mathbf{z}|^{2}}+\min\left(1,C_{1}/2,C_{1}^{2}\right)\frac{(\mathbf{y}-\mathbf{z}).(\mathbf{y}-\bx)}{R^{4}}\\
 &+\max(1,C_{2})\frac{(\mathbf{y}-\mathbf{z}).(\mathbf{y}-\bx)}{R^{2}|\bx-\mathbf{z}|^{2}}
\end{align*}
\begin{align*}
 &=\max(1,C_{2})\left[\frac{(\bx-\mathbf{y}).(\bx-\mathbf{z})}{R^{2}|\bx-\mathbf{z}|^{2}}+\frac{(\mathbf{y}-\mathbf{z}).(\mathbf{y}-\bx)}{R^{4}}+\frac{(\mathbf{y}-\mathbf{z}).(\mathbf{y}-\bx)}{R^{2}|\bx-\mathbf{z}|^{2}}\right]\\
 &+\frac{(\mathbf{y}-\mathbf{z}).(\mathbf{y}-\bx)}{R^{4}}\left[\min\left(1,C_{1}/2,C_{1}^{2}\right)-\max(1,C_{2})\right]
\end{align*}
From there, we proceed as in \cite[Lemma 2.4]{LunRou}. For the first term of the above
\begin{align*}
 &(\bx-\mathbf{y}).(\bx-\mathbf{z})R^{2}+(\mathbf{y}-\mathbf{z}).(\mathbf{y}-\bx)|\bx-\mathbf{z}|^{2}+R^{2}(\mathbf{z}-\bx).(\mathbf{z}-\mathbf{y}) \\
 &=|\bx-\mathbf{z}|^{2}(\mathbf{y}-\mathbf{z}).(\mathbf{y}-\bx)+R^{2}(\bx-\mathbf{z})(\bx-\mathbf{y}-\mathbf{z}+\mathbf{y})\\
 &\leqslant 5R^{2}|\bx-\mathbf{z}|^{2}
\end{align*}
and for the second
\begin{align*}
 \left|\frac{(\mathbf{y}-\mathbf{z}).(\mathbf{y}-\bx)}{R^{4}}\right||\bx-\mathbf{z}|^{2}\le \frac{4}{R^{2}}|\bx-\mathbf{z}|^{2}.
\end{align*}
Comparing to
\begin{align*}
 |\bx-\mathbf{z}|^{2}R^{4}\rho^{-2}&\frac{|\bx-\mathbf{z}|^{2}R^{4}}{|\bx-\mathbf{z}|^{2}+|\bx-\mathbf{y}|^{2}+|\mathbf{y}-\mathbf{z}|^{2}}\\
 &\geqslant \frac{R^{4}|\bx-\mathbf{z}|^{2}}{4R^{2}+|\bx-\mathbf{z}|^{2}}\\
 &\geqslant \frac{R^{2}}{20}|\bx-\mathbf{z}|^{2}
\end{align*}
because $|\bx-\mathbf{z}|\leqslant |\bx-\mathbf{y}|+|\mathbf{y}-\mathbf{z}|\leqslant 4R,$
we conclude that $$\left|S\right|\le C_{5}\rho^{-2}$$

Step 4.
We now study the last case when only one edge is short and the two others long. Let's take $|\bx-\mathbf{y}|\leqslant 2R$ and $|\mathbf{y}-\mathbf{z}|,|\mathbf{z}-\bx|\geqslant 2R$. 
We first notice that the two longer edges have to be of the same order:
\begin{align}
 |\mathbf{y}-\mathbf{z}|\le |\mathbf{y}-\bx|+|\bx-\mathbf{z}|\le 2|\bx-\mathbf{z}|\nn\\
 |\bx-\mathbf{z}|\le |\bx-\mathbf{y}|+|\mathbf{y}-\mathbf{z}| \le 2|\mathbf{y}-\mathbf{z}|\nn\\
 \frac{1}{2}|\bx-\mathbf{z}|\le |\mathbf{y}-\mathbf{z}|\le 2|\bx-\mathbf{z}|.
 \label{L}
\end{align}
We write
\begin{align*}
|\bx-\mathbf{y}|^{2}|\bx-\mathbf{z}|^{2}|\mathbf{y}-\mathbf{z}|^{2}S=\sum_{\bx,\mathbf{y},\mathbf{z}}(\bx-\mathbf{y}).(\bx-\mathbf{z})|\mathbf{y}-\mathbf{z}|^{2}|\bx-\mathbf{y}||\bx-\mathbf{z}|v(\bx-\mathbf{y})v(\bx-\mathbf{z})
\end{align*}
where $|\bx-\mathbf{z}|v(\bx-\mathbf{z})=1$. 
We can now evaluate the right side of the inequality $\eqref{eq:geom three body}$ multiplied by $|\bx-\mathbf{y}|^{2}|\bx-\mathbf{z}|^{2}|\mathbf{y}-\mathbf{z}|^{2}$.
\begin{align*}
 \rho^{-2}|\bx-\mathbf{y}|^{2}|\bx-\mathbf{z}|^{2}|\mathbf{z}-\mathbf{y}|^{2}\ge C_{6}R^{2}\frac{|\bx-\mathbf{z}|^{4}}{|\bx-\mathbf{z}|^{2}+|\mathbf{z}-\mathbf{y}|^{2}}\ge C_{6}R^{2}|\bx-\mathbf{z}|^{2}
\end{align*}
because $|\bx-\mathbf{y}|^{2}\le |\bx-\mathbf{z}|^{2}$.
All that is left is to show that 
\begin{align*}
&(\bx-\mathbf{y}).(\bx-\mathbf{z})|\mathbf{z}-\mathbf{y}|^{2}+(\mathbf{y}-\mathbf{z}).(\mathbf{y}-\bx)|\bx-\mathbf{z}|^{2}\\
 =&|\mathbf{z}-\mathbf{y}|^{2}|\bx-\mathbf{y}|^{2}+|\mathbf{z}-\mathbf{y}|^{2}(\bx-\mathbf{y}).(\mathbf{y}-\mathbf{z})
 +|\bx-\mathbf{z}|^{2}|\bx-\mathbf{y}|^{2}+|\bx-\mathbf{z}|^{2}(\bx-\mathbf{y}).(\mathbf{z}-\bx)\\
 =&|\bx-\mathbf{z}|^{2}|\bx-\mathbf{y}|^{2}+(\bx-\mathbf{y}).(\mathbf{z}-\bx)\left[|\bx-\mathbf{z}|^{2}-|\mathbf{z}-\mathbf{y}|^{2}\right]\\
 =&|\bx-\mathbf{z}|^{2}|\bx-\mathbf{y}|^{2}+(\bx-\mathbf{y}).(\mathbf{z}-\bx)\left[\left(|\mathbf{z}-\bx|+|\mathbf{z}-\mathbf{y}|\right)\left(|\mathbf{z}-\bx|-|\mathbf{z}-\mathbf{y}|\right)\right]\\
 \le & C_{7}R^{2}|\bx-\mathbf{z}|^{2}
\end{align*}
Using the triangular inequality $$\left||\mathbf{z}-\bx|-|\mathbf{z}-\mathbf{y}|\right|\le |\mathbf{z}-\bx-(\mathbf{z}-\mathbf{y})|$$ so we deduce
\begin{align*}
|\bx-\mathbf{y}|^{2}|\bx-\mathbf{z}|^{2}|\mathbf{y}-\mathbf{z}|^{2}S&=|\bx-\mathbf{y}|v(\bx-\mathbf{y})\left [(\bx-\mathbf{y}).(\bx-\mathbf{z})|\mathbf{z}-\mathbf{y}|^{2}+(\mathbf{y}-\mathbf{z}).(\mathbf{y}-\bx)|\bx-\mathbf{z}|^{2}\right ]\\&+(\mathbf{z}-\bx).(\mathbf{z}-\mathbf{y})|\bx-\mathbf{y}|^{2}\le C_{8}R^{2}|\bx-\mathbf{z}|^{2}
\end{align*}
and conclude that
$$\left|S\right|\le C\rho^{-2}.$$

Now we have $$\left|\nabla^{\perp}w_{R}(\mathbf{\bx}_{1}-\mathbf{\bx}_{2}).\nabla^{\perp}w_{R}(\mathbf{\bx}_{1}-\mathbf{\bx}_{3})\right|\leqslant C\left(\left(\mathbf{p}_{1}\right)^{2}+1\right)$$ using Lemma \ref{lem:Hardy}.
Then, as a multiplication operator $\nablap w_{R}$ does not see the phase so
\begin{align*}
 \bral{u} \left|\nabla^{\perp}w_{R}(\mathbf{\bx}_{1}-\mathbf{\bx}_{2}).\nabla^{\perp}w_{R}(\mathbf{\bx}_{1}-\mathbf{\bx}_{3})\right|u\ketr=
 &\bral{|u|} \left|\nabla^{\perp}w_{R}(\mathbf{\bx}_{1}-\mathbf{\bx}_{2}).\nabla^{\perp}w_{R}(\mathbf{\bx}_{1}-\mathbf{\bx}_{3})\right||u|\ketr\\
 \leqslant & C\left(\left|\nabla |u|\right|^{2}+|u|^{2}\right)\\
 \le & C\bral{u},\left (\left(\mathbf{p}_{1}^{\bA}\right)^{2}+1\right )u\ketr
\end{align*}
which gives the expected result using the diamagnetic inequality.

\end{proof}

\subsection{A priori bound for the ground state}\label{sec:ap bound}

For the estimates of the previous subsection to apply efficiently, we need an a priori bound on ground states (or approximate ground states) of 
the $N$-body Hamiltonian~\eqref{HRN}, provided in the following:

\begin{proposition}[\textbf{A priori bound for many-body ground states}]\label{pro:a priori}\mbox{}\\
 Let $\Psi_{N}\in L^{2}_{sym}(\mathbb{R}^{2N})$ be a sequence of approximate ground state for $H_{N}^{R}$ that is:
 $$\bral{\Psi_{N}},H^{R}_{N}{\Psi_{N}}\ketr\leqslant E^{R}(N)(1+o(1))\;\;\text{when}\;\; N\to \infty$$
 Denote by $\gamma_{N}^{(1)}$ the associated sequence of one-body density matrices. In the regime $\alpha=\beta/(N-1)\to 0$, assuming a bound
 $R\geqslant N^{-\eta}$ for some $\eta>0$ independent of $N$, we have:
 \begin{equation}
  \mathrm{Tr}\left[\left(\left(\mathbf{p}^{\bA}\right)^{2}+V\right)\gamma_{N}^{(1)}\right]\leqslant C(1+\beta)\EAF_{R}
  \label{apriorib}
 \end{equation}
\end{proposition}
\begin{proof}
 We proceed in two steps:\newline
Step 1.
 Using a trial state $u^{\otimes N}$ such that $\cEAF_{R}[u]=\EAF_{R}$ we can obtain from $\eqref{energy}$, $\eqref{J}$,$\eqref{ar}$ and from the bounds $\eqref{WRcarre2}$, $\eqref{TBT3}$ and  $\eqref{W3C2}$ plus the diamagnetic inequality $\eqref{eq:af_diamagnetic}$:
 \begin{align}
 \frac{E^{R}(N)}{N} &\leqslant \EAF_{R}-\frac{\beta^{2}}{N-1}\mathrm{Tr}\left[\left(\nabla^{\perp}w_{R}(\mathbf{x}_{1}-\mathbf{x}_{2}).\nabla^{\perp}w_{R}(\mathbf{x}_{1}-\mathbf{x}_{3})\right)\gamma_{N}^{\mathrm{af}(3)}\right]\nn\\
 &+\frac{\beta^{2}}{N-1}\mathrm{Tr}\left[\left|\nabla^{\perp}w_{R}(\mathbf{x}_{1}-\mathbf{x}_{2})\right|^{2}\gamma_{N}^{\mathrm{af}(2)}\right]\nn\\
 &\le \EAF_{R}+\left (\EAF_{R}+1\right )\left (\frac{C\beta^{2}}{N-1}+\frac{C_{\eps}\beta^{2}R^{-\eps}}{N-1}\right )\nn\\
 &=\EAF_{R}+o_{N}(1)\to\EAF
  \label{e1}
  \end{align}
since $R= N^{-\eta}$ with $N\to\infty$, and we used Proposition $\eqref{prop:af_limit}$.

We now use the diamagnetic inequality in each variable to obtain:
\begin{align*}
 \bral{\Psi_{N}},H^{R}_{N}{\Psi_{N}}\ketr&=\sum_{j=1}^{N}\int_{\mathbb{R}^{2N}}\left(\left|(-\im\nabla_{j}+\bA_{e}+\alpha \bA_{j}^{R})\Psi_{N}\right|^{2}+V(\mathbf{x}_{j})|\Psi_{N}|^{2}\right)\mathrm{d}\mathbf{x}_{1}...\mathrm{d} \mathbf{x}_{N}\\
&\geqslant\sum_{j=1}^{N}\int_{\mathbb{R}^{2N}}\left(\left|\nabla_{j}|\Psi_{N}|\right|^{2}+V(\mathbf{x}_{j})|\Psi_{N}|^{2}\right)\mathrm{d}\mathbf{x}_{1}...\mathrm{d} \mathbf{x}_{N}
\end{align*}
We deduce
\begin{equation}
 \mathrm{Tr}\left[\left(\mathbf{p}^{2}+V\right)\gamma_{N,+}^{(1)}\right]\leqslant C\EAF_{R}
\end{equation}
where we denote the reduced k-body density matrix of $|\Psi_{N}|$
$$\gamma_{N,+}^{(k)}=\mathrm{Tr}_{k+1\to N}\left[{\b |\Psi_{N}|}\ketr\bral{|\Psi_{N}|\b }\right]$$
Step 2.
We expand the Hamiltonian and use the Cauchy-Schwarz inequality for operators to get:
 \begin{align*}
  H_{N}^{R}&=\sum_{j=1}^{N}\left(\left(\mathbf{p}_{j}^{\bA}\right)^{2}+\alpha \mathbf{p}_{j}^{\bA}.\bA_{j}^{R}+\alpha \bA_{j}^{R}.\mathbf{p}_{j}^{\bA}+\alpha^{2}|\bA_{j}^{R}|^{2}+V(\mathbf{x}_{j})\right)\\
  &\geqslant\sum_{j=1}^{N}\left((1-2\delta^{-1})\left(\mathbf{p}_{j}^{\bA}\right)^{2}+(1-2\delta)\alpha^{2}|\bA_{j}^{R}|^{2}+V(\mathbf{x}_{j})\right)\\
  &=\sum_{j=1}^{N}\left(\frac{1}{2}(\left(\mathbf{p}_{j}^{\bA}\right)^{2}+V(\mathbf{x}_{j}))-7\frac{\beta^{2}}{(N-1)^{2}}|\bA_{j}^{R}|^{2}\right)
 \end{align*}
choosing $\delta=4$. Thus using $\eqref{e1}$ we have:
\begin{equation}
  \mathrm{Tr}\left[\left(\left(\mathbf{p}^{\bA}\right)^{2}+V\right)\gamma_{N}^{(1)}\right]\leqslant C\EAF_{R}+\frac{C\beta}{N(N-1)^{2}}\left<\Psi_{N},\sum_{j=1}^{N}|\bA_{j}^{R}|^{2}\Psi_{N}\right>
\end{equation}
and since the second term is a purely potential term which does not see the phase:
$$\left<\Psi_{N},\sum_{j=1}^{N}|\bA_{j}^{R}|^{2}\Psi_{N}\right>=\left<|\Psi_{N}|,\sum_{j=1}^{N}|\bA_{j}^{R}|^{2}|\Psi_{N}|\right>.$$
We can now expand the square and use $\eqref{WRcarre2}$ and $\eqref{TBT3}$ to obtain:
\begin{align*}
&\frac{1}{N(N-1)^{2}}\left<|\Psi_{N}|,\sum_{j=1}^{N}|\bA_{j}^{R}|^{2}|\Psi_{N}|\right>\leqslant C\mathrm{Tr}\left[\left(\nabla^{\perp}w_{R}(\mathbf{x}_{1}-\mathbf{x}_{2}).\nabla^{\perp}w_{R}(\mathbf{x}_{1}-\mathbf{x}_{3})\right)\gamma_{N,+}^{(3)}\right]\\
&+CN^{-1}\mathrm{Tr}\left[\left|\nabla^{\perp}w_{R}(\mathbf{x}_{1}-\mathbf{x}_{2})\right|^{2}\gamma_{N,+}^{(2)}\right]\\
&\leqslant C \mathrm{Tr}\left[(\mathbf{p}_{1}^{2}+1)\otimes 1\!\!1\otimes 1\!\!1 \gamma_{N,+}^{(3)}\right]+C_{\epsilon}R^{-\epsilon}N^{-1}\mathrm{Tr}\left[(\mathbf{p}_{1}^{2}+1)\otimes 1\!\!1\gamma_{N,+}^{(2)}\right]\\
 &\leqslant C(1+C_{\epsilon}R^{-\epsilon}N^{-1})\mathrm{Tr}\left[(\mathbf{p}_{1}^{2}+1)\gamma_{N,+}^{(1)}\right]
\end{align*}

\end{proof}

\section{Mean-field limit}
We start the study of the mean-field limit. We use the same strategy as in \cite{LunRou} but
another version of the quantum de Finetti theorem. This version comes from quantum information theory (see \cite{LiSmi-15,BraHar-12}) and provides us with less restrictive constraint for the domain of validity of our main theorem but forces us to put all operators under a tensorized form. We mean: each $N$-particles operator $\Theta$ will be written under the form:
\begin{equation}
\Theta=\int \Theta_{1}(\mathbf{p})\otimes\Theta_{2}(\mathbf{p})...\otimes \Theta_{N}(\mathbf{p})\mathrm{d}P(\mathbf{p})
\label{tensor_form}
\end{equation} 
where each $\Theta_{i}$ is a 1-particle operator.\newline
The rest is similar to what is done in \cite{LunRou} combining with the methods of \cite[Results and discussion]{Rou}.

\subsection{Energy cut-off}
We introduce the spectral projector $P$ of the operator $h=\left(\mathbf{p}^{\bA}\right)^{2}+V$ below an energy cut-off $\Lambda$ to be eventually optimized later.
\begin{equation}
 P=1\!\!1_{h\leqslant\Lambda},\;\;\;\;Q=1\!\!1-P
\end{equation}
We denote $$N_{\Lambda}=\dim \left(PL^{2}(\mathbb{R}^{2})\right)$$ the number of energy levels obtained in this way
and recall the following Cwikel-Lieb-Rozenblum type inequality proved in \cite[Lemma 3.3]{LewNamRou-14}(see \cite{ComSchSei-78,Simon-05} for more details).
\begin{lemma}[\textbf{Number of energy levels below the cut-off}]\label{cut-off}\mbox{}\\
 For $\Lambda$ large enough we have
 \begin{equation}
  N_{\Lambda}\leqslant C\Lambda^{1+2/s}
  \label{3.1}
 \end{equation}

\end{lemma}
\subsection{Localisation in Fock space}
Let $\gamma_{N}$ be an arbitrary $N$-body mixed state.
Associated with the projector $P$, there is a localisated state $G^{P}_{N}$ on the Fock space
$$\mathcal{F}(\gH)=\mathbb{C}\oplus\gH\oplus\gH^{2}\oplus...$$
with the property that its reduced density matrices satisfy
\begin{equation}
 P^{\otimes n}\gamma_{N}^{(n)}P^{\otimes n}=(G_{N}^{P})^{(n)}=\begin{pmatrix}
   N \\
   n
\end{pmatrix}^{-1}\sum_{k=n}^{N}\begin{pmatrix}
   k \\
   n 
\end{pmatrix}\mathrm{Tr}\left[G^{P}_{N,k}\right]
\label{3.4}
\end{equation}
for any $0\leqslant n\leqslant N$.\newline
We use the convention that
$$\gamma_{N}^{(n)}=\mathrm{Tr}_{n+1\to N}\left[\gamma_{N}\right]$$
The localized state $G^{Q}_{N}$ corresponding to the projector $Q$ is defined similarly.
The relations $\eqref{3.4}$ determine the localized states $G^{P}_{N}$  and $G^{Q}_{N}$ 
uniquely ensuring they are mixed states on the projected Fock spaces $\mathcal{F}(P\gH)$
and $\mathcal{F}(Q\gH)$, respectively:
\begin{equation}
 \sum_{k=0}^{N}\mathrm{Tr}\left[G^{P/Q}_{N,k}\right]=1
 \label{Loc}
\end{equation}
The precise construction can be found in \cite[Chapter 5]{Rou-01}.
We will now apply the de Finetti representation of projected density matrices of \cite[Proposition 3.2]{Rou} 
in order to approximate its three-body density matrix. 
This inequality combine Fock-space localisation and the information-theoretic quantum de Finetti theorem (\cite[Appendix A]{Rou}) as we can find in the proof of \cite[Proposition 3.2]{Rou}.


\subsection{Quantum de Finetti theorem}
We call $\mathfrak{S}^{1}$ the set of trace-class operators.
\begin{theorem}[\textbf{de Finetti representation of projected density matrices}]\label{thm:Quantum deF}\mbox{}\\
Let $\gH$ be a complex separable Hilbert space with the corresponding bosonic space $\gH_{N}=\gH^{\otimes_{\mathrm{sym}} N}$.
Let $\gamma_{N}^{(3)}$ be the 3-body reduced density matrix of a N-body state vector $\Psi_{N}\in\gH_{N}$ or a general mixed state in:
$$\mathcal{S}\left(L^{2}_{\mathrm{sym}}\left(\mathbb{R}^{2N}\right)\right)=\{\Gamma\in\mathfrak{S}^{1}\left(L^{2}_{sym}\left(\mathbb{R}^{2N}\right)\right),\Gamma =\Gamma^{*},\Gamma\geqslant 0,\mathrm{Tr}\,\Gamma =1\}$$
Let $P$ be a finite dimensional orthogonal projector.\newline

There exists a positive Borel measure $\mu_{N}^{(3)}$ on the set of one-body mixed states $S_{P}=\mathcal{S}\left(PL^{2}_{\mathrm{sym}}\left(\mathbb{R}^{2N}\right)\right)$ such that for all $A$, $B$ and $C$ hermitian operators:
\begin{align}
 \sup_{A,B,C}\mathrm{Tr}\left|A\otimes B\otimes C\left(P^{\otimes 3}\gamma_{N}^{(3)}P^{\otimes 3}-\int_{S_{P}} \gamma^{\otimes 3}\mathrm{d}\mu_{N}^{(3)}(\gamma)\right)\right|
 \leqslant C\sqrt{\frac{\log (\mathrm{dim}(P))}{N}}||A||.||B||.||C||
 \label{Def}
\end{align}
where the norm $\b\b A\b\b_{\mathrm{op}}$ is the operator norm defined as usual by: $\sup _{\b\b u \b\b=1}\b\b Au\b\b_{L^{2}(\R^{2})}$. 
\end{theorem}

The theorem gives us a better dependency in $\mathrm{dim}(P)$. This comes at the cost of the measure charging general mixed states $\gamma^{\otimes 3}$ instead of  $(\ketl u\ketr \bral u\brar)^{\otimes 3}$. Comparing to \cite[Proposition 3.2]{Rou} we extend the theorem to all hermitian operators.
\begin{proof}
To begin, we call $$\Delta\gamma^{(3)}=\left(P^{\otimes 3}\gamma_{N}^{(3)}P^{\otimes 3}-\int \gamma^{\otimes 3}\mathrm{d}\mu_{N}^{(3)}(\gamma)\right)$$
where the measure $\mu_{N}$ is defined as in \cite{Rou}
and $A=A^{+}+A^{-}=\Pi_{A\geqslant 0}A+\Pi_{A\leqslant 0}A$ with $\Pi$ an orthogonal projector acting on $P\gH$ defined by the projection on the positive and negative part of $A$ (recall that this operator describes an observables and is thus hermitian).
The first term is then
\begin{align*}
 &\sup_{ A,B,C}\mathrm{Tr}\left|A\otimes B\otimes C\left (\Delta\gamma^{(3)}\right )\right \b \\
  = &\sup_{ A,B,C}\mathrm{Tr}\left|\sum_{(i,j,k)\in\{+,-\}}||A^{i}||.||B^{j}||.||C^{k}||\frac{A^{i}\otimes B^{j}\otimes C^{k}\left (\Delta\gamma^{(3)}\right ) }{||A^{i}||.||B^{j}||.||C^{k}||}\right \b \\
   \le & C\sqrt{\frac{\log (dim(P))}{N}}\sum_{(i,j,k)\in\{+,-\}}||A^{i}||.||B^{j}||.||C^{k}||
\end{align*}
using the fact that $\frac{A^{i}}{||A^{i}||}$ is a bounded by $1$ operator which allows us to use \cite[Theorem A.3]{Rou} with the method of \cite[Proposition 3.2]{Rou} but considering three operators instead and using that $||A^{+}||\le||A||$, $||A^{-}||\le||A||$.
\end{proof}


\subsection{Truncated Hamiltonian}
Our aim is now to find an energy lower bound to $H_{N}^{R}$. We introduce  the effective three-body Hamiltonian $\tilde{H}_{3}^{R}$ such that for $||u||=1$:
$$\left< u^{\otimes 3},\tilde{H}_{3}^{R} u^{\otimes 3}\right>=\mathcal{E}_{R}^{\mathrm{af}}[u]\geqslant E^{\mathrm{af}}_{R}$$
and denote
$$ \Tr\left[\tilde{H}_{3}^{R} \gamma^{\otimes 3}\right]=\mathcal{E}_{R}^{\mathrm{af}}[\gamma].$$
Posing 
\begin{equation}
h=\left(\mathbf{p}^{\bA}\right)^{2}+V
\label{h}
\end{equation} we note:
\begin{align}
 \tilde{H}_{3}^{R}&=\frac{1}{3}(h_{1}+h_{2}+h_{3})\nonumber\\
 &+\frac{\beta}{6}\sum_{1\leqslant j \neq k\leqslant 3}\left(\mathbf{p}^{\bA}_{j}.\nabla^{\perp}w_{R}(\mathbf{x}_{j}-\mathbf{x}_{k})+\nabla^{\perp}w_{R}(\mathbf{x}_{j}-\mathbf{x}_{k}).\bp^{\bA}_{j}\right)\nonumber\\
 &+\beta^{2}\nabla^{\perp}w_{R}(\mathbf{x}_{1}-\mathbf{x}_{2}).\nabla^{\perp}w_{R}(\mathbf{x}_{1}-\mathbf{x}_{3})
\end{align}
We denote for shortness
\begin{equation}
W_{2}=\bp_{1}.\nabla^{\perp}w_{R}(\mathbf{x}_{1}-\mathbf{x}_{2})+\nabla^{\perp}w_{R}(\mathbf{x}_{1}-\mathbf{x}_{2}).\bp_{1}
\label{W2}
\end{equation}
 the two-body part and
\begin{equation}
W_{3}=\nabla^{\perp}w_{R}(\mathbf{x}_{1}-\mathbf{x}_{2}).\nabla^{\perp}w_{R}(\mathbf{x}_{1}-\mathbf{x}_{3})
\label{W3}
\end{equation} the three-body part.
Thus
$$\tilde{H}_{3}^{R}=\frac{1}{3}(h_{1}+h_{2}+h_{3})+\frac{\beta}{6}\sum_{1\leqslant j \neq k\leqslant 3}W_{2}(j,k)+\beta^{2}W_{3}$$

We now bound the full energy from below in terms of a projected version of $\tilde{H}_{3}^{R}$.



\begin{lemma}[\textbf{Truncated three-body Hamiltonian}]\label{truncated H}\mbox{}\\
 Let $\Psi_{N}$ be a sequence of approximate ground states for $H_{N}^{R}$ with associated reduced matrices $\gamma_{N}^{(k)}$. Then for $\epsilon>0$ and $R$ small enough:
 \begin{equation}
  N^{-1}\bral{\Psi_{N}},{H_{N}^{R}\Psi_{N}}\ketr\geqslant \mathrm{Tr}\left[\tilde{H}^{R}_{3}P^{\otimes 3}\gamma_{N}^{(3)}P^{\otimes 3}\right]+C\mathrm{Tr}\left[hQ\gamma_{N}^{(1)}Q\right]-C_{\beta}\left(\frac{C_{\epsilon}R^{-\epsilon}}{\sqrt{\Lambda}R}+\frac{C}{\Lambda R^{2}}\right)-C_{\beta}N^{-1}
  \label{LB}
 \end{equation}

\end{lemma}
\begin{proof}
The proof is quite the same as in \cite[Proposition 3.3]{LunRou}. We just have to use our new bounds $\eqref{WRcarre}$,$\eqref{TBT2}$, $\eqref{W3C}$ and 
Proposition $\ref{pro:a priori}$ instead of those of the original paper.
\end{proof}

\subsection{Application of the de Finetti theorem}
We need to apply the de Finetti Theorem~\ref{thm:Quantum deF} to the first term of $\eqref{LB}$ in order to get our energy
$$\int_{S_{P}}\Tr[\tilde{H}_{3}^{R}\gamma^{\otimes 3}]\mathrm{d} \mu_{N}^{(3)}\left (\gamma \right )$$ with a controled error.
\begin{lemma}[\textbf{Error for the truncated Hamiltonian}]\label{err_truncated_H}\begin{equation}
\Tr\left [\tilde{H}_{3}^{R}P^{\otimes 3}\gamma_{N}^{(3)}P^{\otimes 3}\right ]\ge
\int_{S_{P}} \cEAF_{R}[\gamma]d \mu_{N}^{(3)}\left (\gamma \right )- \frac{C_{\beta}}{R^{2}}\left (R^{2}\Lambda+R\sqrt{\Lambda}+1\right )\sqrt{\frac{\log N_{\Lambda}}{N}}
\label{errdef}
\end{equation}
\end{lemma}
\begin{proof}
The only thing we have to do to apply Theorem \ref{thm:Quantum deF} is to put our operators under a tensorized form $\eqref{tensor_form}$. \newline

Step 1.
Our first operator is $h_{i}$ defined in $\eqref{h}$. As it is a one particle observable its tensorized form is simply $$h_{1}\otimes \mathbbm{1}\otimes \mathbbm{1}$$ if $i=1$.\newline

Step 2.
To put $W_{2}$ defined in $\eqref{W2}$ in the form $\eqref{tensor_form}$ we will use the Fourier transform
 $\left (\mathrm{FT}[\;.\;]\right)$ of $\nabla^{\perp}w_{R}(x)$. We can then write

\begin{align*}
W_{2}=-2\pi\im\int\widehat{\chi}_{R}\frac{\be_{p^{\perp}}}{\b \bp\b} e^{-\im \mathbf{p}.\mathbf{x}_{2}}\bp_{1}^{\bA}e^{\im \mathbf{p}.\mathbf{x}_{1}}\mathrm{d}\bp-2\pi\im\int\widehat{\chi}_{R}\frac{\be_{p^{\perp}}}{\b \bp\b} e^{-\im \mathbf{p}.\mathbf{x}_{2}}e^{\im \mathbf{p}.\mathbf{x}_{1}}\bp_{1}^{\bA}\mathrm{d}\bp
\end{align*}
where $ \widehat{\chi}_{R}=\mathrm{FT}[\chi_{R}]$ and
with $\be_{p^{\perp}}$ being the direction of the vector $\bp^{\perp}$ such that $\b\b \be_{p^{\perp}}\b\b =1$.
In order to bring together the two parts of $W_{2}$ we note that
\begin{align*}
\bp_{1}^{\bA}e^{\im \mathbf{p}.\mathbf{x}_{1}}\Psi=(\bp_{1}+\bA_{e})e^{\im \mathbf{p}.\mathbf{x}_{1}}\Psi=\left (\im \bp e^{\im \mathbf{p}.\mathbf{x}_{1}}+e^{\im \mathbf{p}.\mathbf{x}_{1}}\bp_{1}^{\bA}\right )\Psi
\end{align*} 
and obtain
\begin{align}
W_{2}&=-2\im\pi\int\widehat{\chi}_{R}(\mathbf{p})\frac{\be_{p^{\perp}}}{\b \bp\b}e^{-\im \mathbf{p}.\mathbf{x}_{2}}e^{\im \mathbf{p}.\mathbf{x}_{1}}.\left [\im \bp+2\bp_{1}^{\bA}\right ]\mathrm{d}\bp\\
&=-4\im\pi\int\frac{\widehat{\chi}_{R}(\mathbf{p})}{\b \bp\b}e^{-\im \mathbf{p}.\mathbf{x}_{2}}\otimes e^{\im \mathbf{p}.\mathbf{x}_{1}}\be_{p^{\perp}}.\bp_{1}^{\bA}\otimes \id \mathrm{d}\bp
\end{align}
because $\be_{p^{\perp}}. \bp=0$.
The last problem is that $e^{\im px}$ is a complex number and we need hermitian operators.
We rewrite
\begin{align}
\im e^{-\im \mathbf{p}.\mathbf{x}_{2}}e^{\im \mathbf{p}.\mathbf{x}_{1}}&=\im\left[\cos (\mathbf{p}.\mathbf{x}_{1})+\im \sin (\mathbf{p}.\mathbf{x}_{1})\right]\left [\cos (\mathbf{p}.\mathbf{x}_{2}))-\im\sin (\mathbf{p}.\mathbf{x}_{2})\right ]\\
&=\im\left [\cos (\mathbf{p}.\mathbf{x}_{1})\cos (\mathbf{p}.\mathbf{x}_{2})+\sin (\mathbf{p}.\mathbf{x}_{1})\sin (\mathbf{p}.\mathbf{x}_{2})\right ]\\
&+\left 
[\sin (\mathbf{p}.\mathbf{x}_{2})\cos (\mathbf{p}.\mathbf{x}_{1})-\sin (\mathbf{p}.\mathbf{x}_{1})\cos (\mathbf{p}.\mathbf{x}_{2})\right ]
\label{exp}
\end{align}
Here the imaginary part is an even function whereas $\widehat{\chi}_{R}(\mathbf{p})\frac{\be_{p^{\perp}}}{\b \bp\b}$ is odd. Consequently this part of the integral gives zero and we get:
 
\begin{equation}
W_{2}=4\pi\int\frac{\widehat{\chi}_{R}}{\b \bp\b}\left (\cos (\mathbf{p}.\mathbf{x}_{2})\otimes\sin (\mathbf{p}.\mathbf{x}_{1})-\sin (\mathbf{p}.\mathbf{x}_{2})\otimes\cos (\mathbf{p}.\mathbf{x}_{1})\right )\be_{p^{\perp}}.\bp_{1}^{\bA}\otimes \id \mathrm{d}\bp.
\end{equation}

Step 3.
We now consider the three-body operator $W_{3}$ defined in $\eqref{W3}$.
$$W_{3}=\int\widehat{\chi}_{R}(\mathbf{p})\frac{\be_{p^{\perp}}}{\b \bp\b}\im e^{\im \mathbf{p}.(\mathbf{x}_{1}-\mathbf{x}_{2})}\mathrm{d}\bp.\int\widehat{\chi}_{R}(\mathbf{p}')\frac{\be_{p'^{\perp}}}{\b \bp'\b}\im e^{\im \mathbf{p}'.(\mathbf{x}_{1}-\mathbf{x}_{3})}\mathrm{d}\bp'$$
Then, rewriting the exponential term as in $\eqref{exp}$ we obtain the tensorized form:
\begin{align*}
W_{3}&=\int\widehat{\chi}_{R}(\mathbf{p})\widehat{\chi}_{R}(\mathbf{p}')\frac{\be_{p^{\perp}}}{\b \bp\b}.\frac{\be_{p'^{\perp}}}{\b \bp'\b}\\
& \Big[\cos (\mathbf{p}.\mathbf{x}_{1})\cos (\mathbf{p}'.\mathbf{x}_{1})\otimes \sin (\mathbf{p}.\mathbf{x}_{2})
\otimes\sin (\mathbf{p}'.\mathbf{x}_{3})\\
+&\sin (\mathbf{p}.\mathbf{x}_{1})\sin (\mathbf{p}'.\mathbf{x}_{1})\otimes\cos (\mathbf{p}.\mathbf{x}_{2})\otimes \cos (\mathbf{p}'.\mathbf{x}_{3})\\
-& \cos (\mathbf{p}.\mathbf{x}_{1})\sin (\mathbf{p}'.\mathbf{x}_{1})\otimes \sin (\mathbf{p}.\mathbf{x}_{2})\otimes\cos (\mathbf{p}'.\mathbf{x}_{3}) \\
-&\sin (\mathbf{p}.\mathbf{x}_{1})\sin (\mathbf{p}'.\mathbf{x}_{1})\otimes \cos (\mathbf{p}.\mathbf{x}_{2})\otimes\cos (\mathbf{p}'.\mathbf{x}_{3}) 
\Big]\mathrm{d}\bp\mathrm{d}\bp'
\end{align*}
Step 4.
All that is left is to compute the entire Hamiltonian and to apply ($p$ being fixed) Theorem \ref{thm:Quantum deF}. Thus we get
\begin{align}
\left \b \Tr\left [\frac{1}{3}\sum_{i=1}^{3}P^{\otimes 3}\id\otimes h_{i}\otimes\id P^{\otimes 3} \left (\Delta\gamma^{(3)}\right )\right ]\right \b\le C\sqrt{\frac{\log N_{\Lambda}}{N}}\b\b PhP\b\b\le 
 C\Lambda\sqrt{\frac{\log N_{\Lambda}}{N}}
 \label{err1}
 \end{align}
 Recalling that  $$\Delta\gamma^{(3)}=\left(P^{\otimes 3}\gamma_{N}^{(3)}P^{\otimes 3}-\int \gamma^{\otimes 3}\mathrm{d}\mu_{N}^{(3)}(\gamma)\right)$$
We do the same for $W_{2}$ and $W_{3}$ using that $\b\b \cos(x.p)\b\b_{\mathrm{op}}\le 1$ as for the sine, and the triangle inequality
 \begin{align}
 \left \b\Tr\left [\frac{\beta}{6}\sum_{1\le j\neq k\le 3}P^{\otimes 3}W_{2}(i,j)\otimes\id P^{\otimes 3}\left (\Delta\gamma^{(3)}\right )\right ]\right \b&\le  C_{\beta}\sqrt{\frac{\log N_{\Lambda}}{N}}\b\b  P \be_{p^{\perp}}.\bp^{\bA} P\b\b\int \left \b\frac{\widehat{\chi}_{R}(\mathbf{p})}{\bp}\right \b \mathrm{d}\bp\nn\\
 &\le  \frac{C_{\beta}\sqrt{\Lambda}}{R}\sqrt{\frac{\log N_{\Lambda}}{N}}
 \label{err2}
 \end{align}
 where we used
 $$\widehat{\chi}_{R}(\mathbf{p})=\widehat{\chi}(R\bp)$$
 and for the third term
  \begin{align}
 \left \b\Tr\left [\beta^{2}P^{\otimes 3}W_{3} P^{\otimes 3}\left (\Delta\gamma^{(3)}\right )\right ]\right \b\le  C_{\beta}\sqrt{\frac{\log N_{\Lambda}}{N}}\int \left \b\frac{\widehat{\chi}_{R}(\mathbf{p})}{\bp}\frac{\widehat{\chi}_{R}(\mathbf{p}')}{\bp'}\right \b \mathrm{d}\bp\mathrm{d}\bp'\le  \frac{C_{\beta}}{R^{2}}\sqrt{\frac{\log N_{\Lambda}}{N}}
 \label{err3}
 \end{align}
 We now combine $\eqref{err1}$, $\eqref{err2}$ and $\eqref{err3}$ to get $\eqref{errdef}$ by the triangle inequality.

\end{proof}
\subsection{Energy bound}
As we did in $\eqref{e1}$ we test $\eqref{energy}$ against a factorized trial state $\Psi_{N}=\left(u_{R}^{\mathrm{af}}\right)^{\otimes N}$
such as $\mathcal{E}^{\mathrm{af}}_{R}[\Psi_{N}]=E_{R}^{\mathrm{af}}$ to get an energy upper bound. 
\begin{equation}
 \frac{E^{R}(N)}{N}\leqslant \cEAF_{R}[u_{R}^{\mathrm{af}}]+o(1)\rightarrow E^{\mathrm{af}}
 \label{UB}
\end{equation}
as $R=N^{-\eta}$ with $N\rightarrow \infty$ we also used $\eqref{eq:ext to point}$ to get
${E}^{\mathrm{af}}_{R}\rightarrow {E}^{\mathrm{af}}$ when $R\to 0$.

\medskip

We derive the lower bound by using the quantitative quantum de Finetti for localized states Theorem \ref{thm:Quantum deF} on the first term of the estimate $\eqref{LB}$ and we get for any sequence of ground states $\Psi_{N}$ that
\begin{align}
  N^{-1}\bral{\Psi_{N}},{H_{N}^{R}\Psi_{N}}\ketr&\geqslant 
  \int_{S_{P}} \cEAF_{R}[\gamma]\mathrm{d} \mu_{N}^{(3)}\left (\gamma \right )+C\Lambda\mathrm{Tr}\left[Q\gamma_{N}^{1}\right]\nonumber\\
   &- \frac{C_{\beta}}{R^{2}}\left (R^{2}\Lambda+R\sqrt{\Lambda}+1\right )\sqrt{\frac{\log N_{\Lambda}}{N}}\nonumber\\
   &-C_{\beta}\left(\frac{1}{N}+\frac{C_{\epsilon}R^{-\epsilon}}{\sqrt{\Lambda}R}+\frac{1}{\Lambda R^{2}}\right)\nonumber\\
 \label{LB1}
\end{align}
To get this result we used $\eqref{errdef}$ and will use $\eqref{3.1}$ to bound $N_{\Lambda}$.

We now choose $\Lambda$ to minimize the error in $\eqref{LB}$ with $R$ behaving at worst as
$$R= N^{-\eta}$$
Changing a little bit $\eta$ to $\eta-\eps$ we can ignore the $R^{\eps}$ factor.
Our choice of $\Lambda$ has to minimize
$$\left(\frac{C_{\epsilon}}{\sqrt{\Lambda}R}+\frac{1}{\Lambda R^{2}}\right)+\frac{1}{\sqrt{N}}\left (\Lambda +\frac{\sqrt{\Lambda}}{R}+\frac{1}{R^{2}}\right )$$
We see that the limiting terms are $\frac{\Lambda}{\sqrt{N}}$ and $\frac{1}{\Lambda R^{2}}$.
We pick
$\Lambda = N^{2\eta}$ to minimize their sum, with 
$$\eta < \frac{1}{4}$$ 
Under this assumption we can get rid of the $C\mathrm{Tr}\left[Q\gamma_{N}^{1}\right]$ when $\Lambda \to \infty$.
Indeed, we claim that: 
\begin{equation}
\cEAF_{R}[\gamma]\ge 0
\label{cE}
\end{equation} Then it follows from $\eqref{LB1}$ that for a constant $C$ $$C\ge \Lambda\mathrm{Tr}\left[Q\gamma_{N}^{(1)}\right]-C_{\eta}o_{N}(1)$$ and the only way the last inequality remains true when $\Lambda\to\infty$ is
 \begin{equation}
\tr\left[Q\gamma_{N}^{(1)}\right]\xrightarrow[N\to \infty]{}0
\label{eq:tightness}
\end{equation}
To see that $\eqref{cE}$ holds and because $\gamma$ can be a mixed state it is convenient to write its kernel
$$
\gamma (\bx_{1};\bx_{2}) = \sum_j \lambda_j u_j (\bx_{1}) \overline{u_j (\bx_{2})},\;\;\lambda_{i}\ge 0
$$
Then 
\begin{align*}
\tr\left[ \tilde{H}^{R}_{3} \gamma ^{\otimes 3}\right] &= \sum_{i}\lambda_{i}\int_{\R^{2}}h\b u_{i}(\bx_{1})\b^{2}\mathrm{d}\bx
+\beta\sum_{i,j} \lambda_i \lambda_j\iint_{\R^{2}\times\R^{2}}W_{2}\b u_{i}(\bx_{1})\b^{2}\b u_{j}(\bx_{2})\b^{2}\mathrm{d}\bx_{1}\mathrm{d}\bx_{2}\\
&+\beta^{2}\sum_{i,j,k} \lambda_i \lambda_j\lambda_{k}\iiint_{\R^{2}\times\R^{2}\times \R^{2}}W_{3}\b u_{i}(\bx_{1})\b^{2}\b u_{j}(\bx_{2})\b^{2}\b u_{k}(\bx_{3})\b^{2}\mathrm{d}\bx_{1}\mathrm{d}\bx_{2}\mathrm{d}\bx_{3}\\
&=\sum_{i}\lambda_{i}\int_{\R^{2}}\left (\left \b\left (\nabla +i\bA_{e}+ i\beta \bA^{R}\left [\sum_{i}\lambda_{j}\b u_{j}\b^{2}\right ]\right )u_{i}\right \b^{2}+V\b u_{i}\b^{2}\right )\\
&\ge \sum_{i}\lambda_{i}\int\left \b \nabla \b u_{i}\b\right \b^{2}\ge 0
\end{align*}

by $\eqref{eq:af_diamagnetic}$. We conclude that:
$$N^{-1}\bral{\Psi_{N}},{H_{N}^{R}\Psi_{N}}\ketr\geqslant 
  \int_{S_{P}} \cEAF_{R}[\gamma]\mathrm{d} \mu_{N}^{(3)}\left (\gamma \right )+o(1).$$

\subsection{Convergence of states}

We first prove that the de Finetti measures converge. Returning to the proof of \cite[Proposition 3.2]{Rou} we have that

\begin{align}\label{eq:tight 2}
 \int_{\cS_P} \mathrm{d}\mu_N ^{(3)} (\gamma) &= \tr \left( P ^{\otimes 3} \gamma_N^{(3)} P ^{\otimes 3}\right) = \sum_{\ell= 3} ^N \tr\left [G^{P}_{N,\ell}\right ] \frac{\ell (\ell - 1)(\ell-2)}{N(N-1)(N-2)}\nonumber \\
 &\geq  \sum_{\ell= 0} ^N \tr\left [G^{P}_{N,\ell}\right ] \frac{\ell^3 }{N^3} - \frac{C}{N}\geq \left(\sum_{\ell= 0} ^N \tr\left [G^{P}_{N,\ell}\right ] \frac{\ell}{N}\right) ^3 - \frac{C}{N} \nonumber \\
 &= \left( \tr \left( P \gamma_N^{(1)} P \right) \right) ^3 - \frac{C}{N}
\end{align}
where we used Jensen's inequality.\newline 
But~\eqref{eq:tightness} implies 
$$ \tr \left( P \gamma_N^{(1)} P \right) \underset{N\to\infty}{\longrightarrow} 1.$$
so
$$
\int_{\cS_P} \mathrm{d}\mu_N ^{(3)} (\gamma) \underset{N\to\infty}{\longrightarrow} 1.$$
Thus the sequence $(\mu_N^{(3)})_N$ of measures given by Theorem $\eqref{thm:Quantum deF}$ is tight on the set of one-body mixed states
 $$S_{P}=S\left (PL^{2}\left( \R^{2}\right )\right )$$
and modulo a (not-relabeled) subsequence $(\mu_N^{(3)})_N$ converges to a probability measure $\mu $ by a tightness argument
\begin{equation}
\mu_N^{(3)}\to\mu.
\label{CV}
\end{equation}

\medbreak 

\noindent\textbf{Convergence of reduced density matrices.} \newline
Given the previous constructions and energy estimates, the proof of~\eqref{eq:main state} follows almost exactly the one used in~\cite{Rou}.\newline
From Proposition~\ref{pro:a priori} we know that 
$\left (\left (\mathbf{p}^{\bA}\right )^{2} + V\right ) \gamma_N ^{(1)}$ is uniformly bounded in trace-class. 
Under our assumptions, $\left (\left (\mathbf{p}^{\bA}\right )^{2} + V\right)^{-1}$ is compact (see \cite[Theorem 2.7]{AvrHerSim}) and we may thus, 
modulo a further extraction, assume that 
$$ \gamma_{N} ^{(1)} \underset{N\to \infty}{\longrightarrow} \gamma ^{(1)}$$
strongly in trace-class norm. 
Modulo extraction of subsequences we have 
$$\gamma_{N} ^{(k)} \wto_* \gamma ^{(k)}$$
weakly-$*$ in the trace-class as $N\to \infty$. 
So by \cite[Corollary 2.4]{LewNamRou-13} we have
$$\gamma_{N} ^{(k)} \to \gamma ^{(k)}$$
strongly in class-trace norm for $k\ge 0$.

Applying the weak quantum de Finetti theorem~\cite[Theorem~2.2]{LewNamRou-15} we deduce that there exists a measure $\nu$ on the unit ball of $L^2 (\R^2)$ such as
$$
\gamma ^{(k)} = \int |u ^{\otimes k } \rangle \langle u ^{\otimes k}| \mathrm{d}\nu (u).
$$
But since $\gamma ^{(1)}$ must have trace $1$, the measure $\nu$ must actually live on 
$$ S L^2 (\R ^d) = \left\{ u\in L^2 (\R^2),\: \int_{\R^2} |u| ^2 = 1 \right\},$$
the unit sphere of $L^2 (\R^2)$. 

Next we claim that the two measures $\mu$ (of $\eqref{CV}$) and $\nu$ just found are related by 
\begin{equation}\label{eq:pouet}
 \int |u ^{\otimes 3 } \rangle \langle u ^{\otimes 3}| \mathrm{d}\nu (u) = \int |u ^{\otimes 3 } \rangle \langle u ^{\otimes 3}| \mathrm{d}\mu (u). 
\end{equation}
Indeed, let 
$$\tilde{P} = \1_{h\leq \tilde{\Lambda}}$$
where $\tilde{\Lambda}$ is a fixed cut-off (different from $\Lambda$ above). Testing~\eqref{Def} with $A_1,A_2, A_{3}$ finite rank operators whose ranges lie within that of $\tilde{P}$ we get
$$
\tr \left( A_1 \otimes A_2\otimes A_{3} \gamma_{N}^{(3)}\right) \underset{N\to\infty}{\longrightarrow} \tr \left( A_1 \otimes A_2 \otimes A_{3}\int_{\cS_{\tilde{P}}} \gamma ^{\otimes 3} \mathrm{d}\mu (\gamma) \right)
$$
using the convergence of $\mu_{N} ^{(3)}$ to $\mu$. On the other hand, by the convergence of $\gamma_{N} ^{(3)}$ to $\gamma^{(3)}$ we also have 
\begin{align*}
\tr \left( A_1 \otimes A_2 \otimes A_{3}\gamma_{N}^{(3)}\right) \underset{N\to\infty}{\longrightarrow} \tr \left( A_1 \otimes A_2 \otimes A_{3}\gamma ^{(3)} \right)\\
 =  \tr \left( A_1 \otimes A_2\otimes A_{3} \int_{S L^2 (\R^2)} |u ^{\otimes 3} \rangle \langle u^{\otimes 3}| \mathrm{d}\nu (u)\right)
\end{align*}
Thus 
\begin{equation}\label{eq:id meas}
 \tr \left( A_1 \otimes A_2 \otimes A_{3}\int_{\cS_{\tilde{P}}} \gamma ^{\otimes 3} \mathrm{d}\mu (\gamma) \right) = \tr \left( A_1 \otimes A_2\otimes A_{3} \int_{S L^2 (\R^2)} |u ^{\otimes 3} \rangle \langle u^{\otimes 3}| \mathrm{d}\nu (u)\right) 
\end{equation}
for any $A_1,A_2, A_{3}$ with range within that of $\tilde{P}$. Letting finally $\tilde{\Lambda} \to \infty$ yields $\tilde{P} \to \1$ and thus~\eqref{eq:id meas} holds for any compact operators $A_1,A_2,A_{3}$. This implies~\eqref{eq:pouet}. In particular, since the left-hand side of~\eqref{eq:pouet} is $\gamma ^{(3)}$, a bosonic operator,  $\mu$ must be supported on pure states $\gamma = |u\rangle \langle u|$, see~\cite{HudMoo-75}.  \newline

\noindent\textbf{Final passage to the liminf}. Let us return to~\eqref{LB1}. We split the integral over one-body states $\gamma$ between low and high kinetic energy states:
$$ \cS^{-} = \left\{ \gamma \in \cS, \, \tr \left( h \gamma \right) \leq \mathrm{L} \right\}, \quad \cS^{+} = \cS \setminus \cS^{-}. $$
Using Lemma~\ref{pro:a priori} we obtain 
\begin{align*}
\int_{\cS_P} \cEAF_{R}[\gamma] \mathrm{d}\mu_N ^{(3)} (\gamma) &\geq  \mathrm{L} \int_{\cS^{+}} \mathrm{d}\mu_N ^{(3)} (\gamma) + \int_{\cS^{-}} \cEAF_{R}[\gamma] \mathrm{d}\mu_N ^{(3)} (\gamma)  - o_{N}(1)\\
&\geq \int_{\cS_P} \min\left(  \mathrm{L}, \cEAF_{R}[\gamma] \right) \mathrm{d}\mu_N ^{(3)}(\gamma) - o_{N}(1) 
\end{align*}
Inserting in~\eqref{LB1} and passing to the liminf in $N\to \infty$ this implies 
$$ \EAF \geq \underset{N\to \infty}{\liminf} \frac{E(N)}{N} \geq \int_{\cS} \min\left( \mathrm{L}, \cEAF[\gamma] \right)  \mathrm{d}\mu(\gamma).$$
Finally, we pass to the limit $\mathrm{L} \to \infty$ to deduce 
\begin{equation}\label{eq:final low}
 \EAF \geq \underset{N\to \infty}{\liminf} \frac{E(N)}{N} \geq \int_{\cS} \cEAF[\gamma] \mathrm{d}\mu(\gamma).
\end{equation}
But as we saw above $\mu$ must be supported on pure states $\gamma = |u\rangle \langle u |$
 which yields both the energy lower bound concluding the proof of~\eqref{eq:main ener} and the fact that $\mu$ must be supported on $\cMAF$. Because $\cEAF [\gamma]$ is a linear function of $\gamma ^{\otimes 3}$ we can also combine~\eqref{eq:final low} with~\eqref{eq:pouet} to deduce that also $\nu$ must be supported on $\cMAF$, which proves~\eqref{eq:main state} and
 $$ \EAF \geq \underset{N\to \infty}{\liminf} \frac{E(N)}{N} \geq \EAF.$$

\medbreak  

\appendix

\section{Properties of the average-field functional} \label{app:AFF}

In this appendix we etablish some of the fundamental properties of the
functional $\cEAF_{R}$ and its limit $R \to 0$.
We call $\nabla_{\bA_{e}}=\nabla +\im \bA_{e}$ the covariant derivative.
For $\beta \in \R$, $\bA_{e}\in L^{2}_{loc}\left (\R^{2}\right )$ and $V: \R^2 \to \R^+$ we define 
the average-field energy functional
\begin{equation}\label{eq:AFF}
	\cEAF[u] := \int_{\R^2} \left( \left| \left( \nabla_{\bA_{e}} + \im \beta \bA[|u|^2] \right) u \right|^2 + V|u|^2 \right),
\end{equation}
with the self-generated magnetic potential
$$
	\bA[\rho] := \nablap w_0 * \rho = \int_{\R^2} \frac{(\bx_{1}-\bx_{2})^\perp}{|\bx_{1}-\bx_{2}|^2} \rho(\bx_{2}) \,\mathrm{d}\bx_{2},
	\qquad \curl\:\bA[\rho] = 2\pi\rho.
$$

\begin{lemma}[\textbf{Bound on the magnetic term}]\label{lem:af_three_body}\mbox{}\\
	We have for any $u \in L^2(\R^2)$ that
	$$
		\int_{\R^2} \left| \bA[|u|^2] \right|^2 |u|^2 
		\le \frac{3}{2} \|u\|_{L^2(\R^2)}^4 \int_{\R^2} \left| \nabla |u| \right|^2
		\le \frac{3}{2} \|u\|_{L^2(\R^2)}^4 \int_{\R^2} \left| \nabla_{\bA_{e}} u \right|^2
		$$
\end{lemma}

\begin{proof}
See \cite[Lemma A.1]{LunRou} combined with $\eqref{eq:af_diamagnetic_simple}$.
\end{proof}

Using the same reasonning than in \cite[Equation A2]{LunRou} we can, using $\eqref{lem:af_three_body}$ define the domain of $\cEAF$ to be
(see \cite{LieLos-01}[Proposition 7.20]) and otherwise let $\cEAF[u] := +\infty$
$$
	\cDAF := \left\{ u \in H^1_{\bA_{e}}(\R^2) : \int_{\R^2} V|u|^2 < \infty \right\},
$$

We find using Cauchy-Schwarz and Lemma~\ref{lem:af_three_body},
 that for $u \in \cDAF$
\begin{align*}
0 \le \cEAF[u] 
	&\le 2\|\nabla_{\bA_{e}} u\|^2 + 2\|\ \bA_{e}u\|^2+3\beta^2 \| \bA[|u|^2] u \|^2 + \int V|u|^2\\
	&\le (2 + 3\beta^2 \|u\|^4) \|\nabla_{\bA_{e}} u\|^2 + C\int V|u|^2 < \infty.
\end{align*}

The ground-state energy of the average-field functional is then given by
$$
	\EAF := \inf \left\{ \cEAF[u] : u \in \cDAF, \int_{\R^2} |u|^2 = 1 \right\}.
$$
For convenience we also make the assumption on $V$ that 
$V(x) \to +\infty$ as $|x| \to \infty$.
Note that $C_c^\infty(\R^2) \subseteq \cDAF$ is dense in $H_{\bA_{e}}^{1}(\R^{2})$.

\begin{lemma}[\textbf{Basic magnetic inequalities}]\label{lem:af_smooth_ineqs}\mbox{}\\
	We have for $u \in \cDAF$ that (diamagnetic inequality)
	\begin{equation}\label{eq:af_diamagnetic_simple}
	\int_{\R^2} \left| (\nabla +\im\bA_{e})u \right|^2 
	\ge \int_{\R^2} \left| \nabla |u| \right|^2
	\end{equation}
\begin{equation}\label{eq:af_diamagnetic}		
		\int_{\R^2} \left| (\nabla +\im\bA_{e}+ \im\beta \bA[|u|^2])u \right|^2 
		\ge \int_{\R^2} \left| \nabla |u| \right|^2,
\end{equation}
	and
	\begin{equation}\label{eq:af_lower_bound}
		\int_{\R^2} \left| (\nabla +\im\bA_{e}+ \im\beta \bA[|u|^2])u \right|^2 
		\ge \int_{\R^{2}} \bB |u|^2+2\pi|\beta| \int_{\R^2} |u|^4.
	\end{equation}
\end{lemma}
\begin{proof}
This is an application of [Lemma 3.2]\cite{CorLunRou-16}.
	
\end{proof}

\begin{proposition}[\textbf{Existence of minimizers}]\label{prop:af_minimizer}\mbox{}\\
	For any value of $\beta \in \R$
	there exists $\uAF \in \cDAF$ with $\int_{\R^2} |\uAF|^2 = 1$ and
	$\cEAF[\uAF] = \EAF$.
\end{proposition}
\begin{proof}
	First note that for $u \in \cDAF$, by Lemma~\ref{lem:af_three_body}
	and Lemma~\ref{lem:af_smooth_ineqs},
	\begin{align*}
		\norm{\nabla_{\bA_{e}} u}_2 &= \norm{\nabla_{\bA_{e}} u + \im\beta \bA[|u|^2] u - \im\beta \bA[|u|^2] u} _2
		\le \cEAF[u]^{1/2} + |\beta| \norm{\bA[|u|^2]u}_2 \\
		&\le \cEAF[u]^{1/2} + |\beta| \sqrt{\frac{3}{2}} \norm{u}_2^2 \norm{\nabla|u|}_2 
		\le \left( 1 + |\beta| \sqrt{\frac{3}{2}} \norm{u}_2^2 \right) \cEAF[u]^{1/2}.
	\end{align*}
	Now take a minimizing sequence 
	$$(u_n)_{n \to \infty} \subset \cDAF,\:
	\norm{u_n}_2 = 1,\: 
	\lim_{n \to \infty} \cEAF[u_n] = \EAF.$$
	Then clearly $(u_n)$ is uniformly bounded in both $L^2(\R^2)$,
	$L^2_V$, and $H_{\bA_{e}}^1(\R^2)$ 
	and therefore  there exists $\uAF \in \cDAF$ 
	and a weakly convergent subsequence (still denoted $u_n$) such that
	$$
		u_n \rightharpoonup \uAF \ \text{in} \ L^2(\R^2) \cap L^2_V \cap H^{1}_{\bA_{e}}(\R^2), \quad
		\nabla_{\bA_{e}} u_n \rightharpoonup \nabla_{\bA_{e}} \uAF \ \text{in} \ L^2(\R^2).
	$$
	Moreover, since $(\left (\mathbf{p}^{\bA}\right )^{2} + V + 1)^{-1/2}$ is compact we have that
	$$u_n = \left (\left (\mathbf{p}^{\bA}\right )^{2}+V + 1\right )^{-1/2}\left (\left (\mathbf{p}^{\bA}\right )^{2}+V + 1\right )^{1/2}u_n$$ 
	is actually strongly convergent (again, extracting a subsequence), 
	hence 
	$$u_n \to \uAF \mbox{ in } L^2(\R^2).$$
It then follows that

	\begin{align*}
	\norm{ (\nabla_{\bA_{e}} + \im\beta \bA[|u|^2])u }^{2}_2=  \bral{(\nabla_{\bA_{e}} + \im\beta \bA[|u|^2])u},(\nabla_{\bA_{e}} + \im\beta \bA[|u|^2])u \ketr
	\end{align*}
	with
	\begin{align*}
	&\liminf_{n \to \infty}\bral{\nabla_{\bA_{e}}u_{n}},\nabla_{\bA_{e}}u_{n}\ketr \ge\int\b \nabla_{\bA_{e}}u\b^{2} \\
	&\lim_{n \to \infty}\bral{\nabla_{\bA_{e}}u_{n}},\bA[|u_{n}|^2]u_{n}\ketr =\bral{\nabla_{\bA_{e}}u},\bA[|u|^2]u\ketr\\
	&\lim_{n \to \infty}\bral{ \bA[|u_{n}|^2]u_{n}}, \bA[|u_{n}|^2]u_{n}\ketr =\int \b \bA[|u^{2}|]u\b^{2}
	\end{align*}
by the weak lower semicontinuity of the norm, the weak-strong convergence and the fact that the functions $\bA[|u_n|^2]u_n$ are strongly converging 
	to $\bA[|u|^2]u$ in $L^2(\R^2)$, see \cite[Proposition A.3]{LunRou}.
	And since $\norm{\cdot}_{L^2_V}$ is also weakly lower semicontinuous
	(see, e.g., \cite[Supplement to IV.5]{ReeSim1}),
	we have $\liminf_{n \to \infty} \cEAF[u_n] \ge \cEAF[\uAF]$.
	Thus, with $\|\uAF\| = \lim_{n \to \infty} \|u_n\| = 1$,
	we also have $\cEAF[\uAF] = \EAF$.
\end{proof}

\begin{proposition}[\textbf{Regularized functional}] \label{prop:_limit}\mbox{}\\
Let us now consider the corresponding situation for the regularized functional
(extended anyons)
$$
	\cEAF_R[u] := \int_{\R^2} \left( \left| \left( \nabla +\im\bA_{e}+ \im \beta \bA^R[|u|^2] \right) u \right|^2 + V|u|^2 \right),
	\quad \bA^R[\rho] := \nablap w_R * \rho,
	\quad R > 0.
$$
Since $\nabla w_R \in L^\infty(\R^2)$ we have 
$\bA^R[|u|^2] \in L^\infty(\R^2)$ 
with 
$$\norm{\bA^R[|u|^2]}_\infty \le \frac{C}{R}\|u\|_2^2$$
and instead of Lemma~\ref{lem:af_three_body} we have
$$\norm{\bA^R[|u|^2]u}_2 \le C\|u\|_2^2 \||u|\|_{H_{\bA_{e}}^1}$$ 
using Lemma~\ref{lem:three body}.

\end{proposition}
Hence the natural domain is again $\cDAF$
and all properties established above 
for $\cEAF$ are also found to be valid for $\cEAF_R$ 
(except $\eqref{eq:af_lower_bound}$
which now has regularized versions).
Denoting 
$$\EAF_R := \min \{\cEAF_R[u] : u \in \cDAF, \|u\|_2 = 1\},$$
we furthermore have the following relationship:

\begin{proposition}[\textbf{Convergence to point-like anyons}] \label{prop:af_limit}\mbox{}\\
	The functional $\cEAF_R$ converges pointwise to $\cEAF$ as $R \to 0$. More precisely, for any $u\in \cDAF$
	\begin{equation}\label{eq:ext to point}
	 \left| \cEAF_R [u] - \cEAF [u] \right| \leq 
	C_u \beta^{2}\cEAF[u]^{1/2}(1+\cEAF[u])^{1/2} R,
	\end{equation}
	where $C_u$ depends only on $\norm{u}_2$. 
	Hence, 
	$$\EAF_R \underset{R\to 0}{\to} \EAF,$$  
	and if $(u_R)_{R \to 0} \subset \cDAF$ denotes 
	a sequence of minimizers of $\cEAF_R$, then there exists a
	subsequence $(u_{R'})_{R' \to 0}$ s.t. $u_{R'} \to \uAF$ as $R' \to 0$,
	where $\uAF$ is some minimizer of $\cEAF$.
\end{proposition}
\begin{proof}
This is an easy variant of the proof of \cite[Proposition A.5]{LunRou}.
\end{proof}

\begin{proposition}[\textbf{Link with the regularized functional}] \label{dev}\mbox{}\\
The Hamiltonian $\eqref{HRN}$ can be expanded (with the notation $\mathbf{p}_{j}+\bA_{e}=\mathbf{p}_{j}^{\bA}$) and corresponds to the energy:
 \begin{align}
 N^{-1}&\bral{\Psi_{N}},H_{R}^{N}{\Psi_{N}}\ketr=\mathrm{Tr}\left[\left((\mathbf{p}_{j}^{\bA})^{2}+V(\mathbf{x}_{j})\right)\gamma_{N}^{(1)}\right]\nonumber\\
 &+\beta\mathrm{Tr}\left[\left(\mathbf{p}_{1}^{\bA}.\nabla^{\perp}w_{R}(\mathbf{x}_{1}-\mathbf{x}_{2})+\nabla^{\perp}w_{R}(\mathbf{x}_{1}-\mathbf{x}_{2}).\mathbf{p}_{1}^{\bA}\right)\gamma_{N}^{(2)}\right]\nonumber\\
 &+\beta^{2}\frac{N-2}{N-1}\mathrm{Tr}\left[\left(\nabla^{\perp}w_{R}(\mathbf{x}_{1}-\mathbf{x}_{2}).\nabla^{\perp}w_{R}(\mathbf{x}_{1}-\mathbf{x}_{3})\right)\gamma_{N}^{(3)}\right]\nonumber\\
 &+\beta^{2}\frac{1}{N-1}\mathrm{Tr}\left[\left|\nabla^{\perp}w_{R}(\mathbf{x}_{1}-\mathbf{x}_{2})\right|^{2}\gamma_{N}^{(2)}\right]
 \label{energy}
 \end{align}
In particular, taking a trial state $\psi_{N}=u^{\otimes N}$ we observe that:
$$\mathrm{Tr}\left[\left(\bp_{1}^{\bA}.\nabla^{\perp}w_{R}(\mathbf{x}_{1}-\mathbf{x}_{2})+\nabla^{\perp}w_{R}(\mathbf{x}_{1}-\mathbf{x}_{2}).\bp_{1}^{\bA}\right)\ketl u^{\otimes 2}\ketr\bral u^{\otimes 2}\brar\right]$$
\begin{align}
 &=\im\iint_{\mathbb{R}^{2}\times\mathbb{R}^{2}}\nabla \overline{u}(\bx_{1})\overline{u}(\bx_{2}).\nabla^{\perp}w_{R}(\mathbf{x}_{1}-\mathbf{x}_{2})u(\bx_{1})u(\bx_{2})\mathrm{d}\bx_{1}\mathrm{d}\bx_{2}\\
 &\;\;\;\;\;\;\;\;\;\;\;\;\;\;\;\;\;\;\;\;\;\;\;\;\;\;\;\;\;\;\;\;\;\;\;\;\;\;\;\;\;\;\;\;\;+\iint_{\mathbb{R}^{2}\times\mathbb{R}^{2}}\bA_{e}(\bx_{1}).\nabla^{\perp}w_{R}(\mathbf{x}_{1}-\mathbf{x}_{2})|u(\bx_{1})|^{2}|u(\bx_{2})|^{2}\mathrm{d}\bx_{1}\mathrm{d}\bx_{2}\nonumber\\
 &-\im\iint_{\mathbb{R}^{2}\times\mathbb{R}^{2}} \overline{u}(\bx_{1})\overline{u}(\bx_{2}).\nabla^{\perp}w_{R}(\mathbf{x}_{1}-\mathbf{x}_{2})\nabla u(\bx_{1})u(\bx_{2})\mathrm{d}\bx_{1}\mathrm{d}\bx_{2}\\
 &\;\;\;\;\;\;\;\;\;\;\;\;\;\;\;\;\;\;\;\;\;\;\;\;\;\;\;\;\;\;\;\;\;\;\;\;\;\;\;\;\;\;\;\;\;+\iint_{\mathbb{R}^{2}\times\mathbb{R}^{2}}\bA_{e}(\bx_{1}).\nabla^{\perp}w_{R}(\mathbf{x}_{1}-\mathbf{x}_{2})|u(\bx_{1})|^{2}|u(\bx_{2})|^{2}\mathrm{d}\bx_{1}\mathrm{d}\bx_{2}\nonumber\\
& =\iint_{\mathbb{R}^{2}\times\mathbb{R}^{2}} |u(\bx_{2})|^{2}\nabla^{\perp}w_{R}(\mathbf{x}_{1}-\mathbf{x}_{2}).\left(\im\left [u(\bx_{1})\nabla \overline{u}(\bx_{1})-\overline{u}(\bx_{1})\nabla u(\bx_{1})\right ]+2\bA_{e}(\bx_{1})|u(\bx_{1})|^{2}\right)\mathrm{d}\bx_{1}\mathrm{d}\bx_{2}\nonumber\\
& =2\int_{\mathbb{R}^{2}}\bA^{R}[|u|^{2}].\left(\bJ[u(\bx_{1})]+\bA_{e}(\bx_{1})|u(\bx_{1})|^{2}\right)\mathrm{d}\bx_{1}
\label{J}
\end{align}
where $$\bJ[u]=\frac{\im}{2}(u\nabla\overline{u}-\overline{u}\nabla u)$$  $$\bJ^{\mathrm{tot}}[u]=\bJ[u]+\bA_{e}|u|^{2}$$ with
$$
 \b\bA^{R}[|u|^{2}]=\int_{\mathbb{R}^{2}}|u(\bx_{2})|^{2}\nabla^{\perp}w_{R}(\mathbf{x}_{1}-\mathbf{x}_{2})\mathrm{d}\bx_{2}
$$
for the second term and:
$$\mathrm{Tr}\left[\left(\nabla^{\perp}w_{R}(\mathbf{x}_{1}-\mathbf{x}_{2}).\nabla^{\perp}w_{R}(\mathbf{x}_{1}-\mathbf{x}_{3})\right)\gamma_{N}^{(3)}\right]$$
\begin{align}
 =&\iiint_{\mathbb{R}^{2}\times\mathbb{R}^{2}\times\mathbb{R}^{2}}|u(\bx_{1})|^{2}|u(\bx_{2})|^{2}|u(\bx_{3})|^{2}.\nabla^{\perp}w_{R}(\mathbf{x}_{1}-\mathbf{x}_{2}).\nabla^{\perp}\mathrm{d}\bx_{1}\mathrm{d}\bx_{2}\mathrm{d}\mathbf{x}_{3}\nonumber\\
 =&\int_{\mathbb{R}^{2}}|u|^{2}\left|\bA^{R}[|u|^{2}]\right|^{2}\mathrm{d}\bx_{1}
 \label{ar}
\end{align}
for the third.
\end{proposition}
\bibliographystyle{siam}
\bibliography{biblio-TG_Oct19}
\end{document}